\documentclass[11pt]{article}

\usepackage{amsmath,amssymb}

\usepackage{amssymb}
\usepackage{amscd}
\usepackage{amsmath}
\usepackage{amsfonts}
\usepackage{amsthm}

\usepackage{ifpdf}   

\ifpdf  
\usepackage{graphicx,color}
\else

\fi

\newtheorem{theo}{Theorem}[section]
\newtheorem{cor}[theo]{Corollary}
\newtheorem{prop}[theo]{Proposition}
\newtheorem{lem}[theo]{Lemma}

\theoremstyle{remark}

\newcommand{\R}{{\mathbb{R}}}
\newcommand{\N}{{\mathbb{N}}}
\newcommand{\C}{{\mathbb{C}}}
\newcommand{\Z}{{\mathbb{Z}}}
\newcommand{\Q}{{\mathbb{Q}}}

\newcommand{\la}{\lambda} 
\newcommand{\al}{\alpha} 
\newcommand{\be}{\beta} 
\newcommand{\Om}{\Omega} 
 
\newcommand{\ga}{\gamma}
\newcommand{\om}{\omega} 
\newcommand{\Si}{\Sigma} 
\newcommand{\si}{\sigma} 
\newcommand{\Ga}{\Gamma}
\newcommand{\De}{\Delta} 
 
\newcommand{\T}{\mathbb{T}}
\newcommand{\op}{\operatorname}
\newcommand{\mo}{\mathcal{M}}

\newcommand{\Ci}{\mathcal{C}^{\infty}}

\newcommand{\con}{\overline} 
\newcommand{\bigo}{\mathcal{O}}

\newcommand{\id}{\operatorname{id}}

\begin{document}
\author{Laurent Charles \footnote{Institut de
    Math{\'e}matiques de Jussieu-Paris give gauche (UMR 7586), Universit{\'e} Pierre et
    Marie Curie, Paris, F-75005 France.}}
 
\title{On the Witten asymptotic conjecture for Seifert manifolds} 

\maketitle

\begin{abstract} 
We give a new proof of Witten asymptotic conjecture for Seifert manifolds with non vanishing Euler class and one exceptional fiber. Our method is based on semiclassical analysis on a two dimensional phase space torus. We prove that the Witten-Reshetikhin-Turaev invariant of a Seifert manifold is the scalar product of two Lagrangian states, and we estimate this scalar product in the large level limit. The leading order terms of the expansion are naturally given in terms of character varieties, the Chern-Simons invariants and some symplectic volumes. For the analytic part, we establish some singular stationary phase lemma for discrete oscillatory sums. 
\end{abstract}

Witten asymptotic conjecture sets up a bridge between quantum and classical invariants. It predicts that the Witten-Reshetikhin-Turaev invariant of a closed 3-manifold $M$ for a given group compact $G$ has an asymptotic expansion in the large level limit, the leading terms of this expansion being function of the Chern-Simons invariant, the Reidemeister torsion and the spectral flow of the representation of $\pi_1(M)$ into $G$. This conjecture has been recently settled for some hyperbolic manifolds \cite{LJ1}, \cite{LJ2} and \cite{oim_MCG}. It was also checked for many Seifert manifolds \cite{FrGo}, \cite{Je1}, \cite{Ro2}, \cite{Ro}, \cite{LaRo}, \cite{Han}, \cite{HaTa}, \cite{hi2}, \cite{BeWi}, \cite{An}, \cite{AnHi}

In this paper, we come back to the Seifert case and propose a new approach, more conceptual than the previous ones. 
In particular, the character variety of $M$ and the Chern-Simons invariant appear naturally in the proof. Furthermore, we compute the leading terms of the amplitudes corresponding to irreducible representations as symplectic volumes of some moduli spaces, which is a new result. In the companion paper \cite{LL}, it is proved that these symplectic volumes are actually integrals of Reidemeister torsion.

To explain briefly our strategy, assume that $M$ is obtained by gluing some solid tori $T_1$, \ldots, $T_n$ to $\Si \times S^1$ where $\Si$ is a compact connected orientable surface with $n$ boundary components. The WRT invariant of $\Si \times S^1$ is an element of a vector space associated to the boundary $\partial \Si \times S^1$. This vector space may be viewed as the quantization of the character variety of $ \partial \Si \times S^1$. We will deduce from Verlinde formula and Riemann-Roch theorem that the WRT invariant of $\Si \times S^1$ is a Lagrangian state supported by the character variety of $\Si \times S^1$. Similarly, it is known that the WRT invariant of $T_1 \cup \ldots \cup T_n$ is also a Lagrangian state. So the WRT invariant of $M$ is the scalar product of two Lagrangian states. This can be estimated by using some pairing formula, when the Lagrangian submanifolds supporting the states intersect transversally. This transversality assumption is satisfied when $M$ has a non vanishing Euler number. Unfortunately, some difficulties arise because the character varieties of $\Si \times S^1$ has singularities, which lead us to estimate some singular oscillatory sums. 
For the sake of simplicity, we restrict ourselves in this paper to Seifert manifolds with one exceptional fiber and the group $\op{SU}(2)$.


\subsubsection*{Statement of the main result}

Let $S$ be the oriented Seifert manifold with unnormalized invariant $(g$; $(a,b))$, where $g$ is an integer $\geqslant 2$ and $a$, $b$ are two coprime integers such that $a\neq0 $ and $b \geqslant 1$. Recall that $S$ is obtained from a genus $g$ oriented surface $\Si$ with boundary a circle $C$, by gluing a solid torus $T$ to $\Si \times S^1$ in such a way that the homology class of a meridian of the boundary of $T$ is sent to $a [C ] + b [S^1]$.

Let $\mo (S)$ be the space of conjugacy classes of group morphisms from $\pi_1 (S)$ to $\op{SU}(2)$. Let $X = \pi_0 (\mo (S))$ be the set of connected components. Introduce the functions $\al , \be : \mo (S) \rightarrow [0, \pi ]$ given by 
$$ \al ( [\rho] )  =  \arccos \bigl( \tfrac{1}{2} \op{tr} ( \rho (C)) \bigr), \quad  \be ([ \rho] )  =  \arccos \bigl( \tfrac{1}{2} \op{tr} ( \rho (S^1)) \bigr) $$
These functions are actually constant on each component of $\mo (S)$. We denote again by $\al$ and $\be$ the induced maps from $X$ to $[0, \pi ]$. They allow to distinguish between the components, that is the joint map $(\al, \be) : X \rightarrow [0, \pi ]^2$ is one-to-one.

 We say that a component of $\mo (S)$ is abelian if it contains only abelian representations, and irreducible otherwise. 
We will divide $X$ into 4 disjoint subsets: $X_1 $ is the set of abelian components, whereas $X_2$, $X_3$, $X_4$ are the sets of irreducible components $x$ such that $\al (x) \neq 0,1$,  $\al (x) = 0$, $\al ( x) =1$ respectively. All the components in the same $X_i$ have the same dimension. In table \ref{tab:xi}, we indicate for each $i$ the possible values of $\al$ and $\be$ on $X_i$, the cardinal of $X_i$ and the dimension of its elements. 

\begin{table} 
\renewcommand\arraystretch{1.1}
\begin{tabular}{|l|c|c|c|c|c|}  
\hline 
  & $\al$ & $\be$ & type & $\#$ & $\frac{1}{2}$dimension \\
\hline $X_1$ & 0 & $\neq 0,\pi$ & abelian &  $\op{E} ( \frac{b-1}{2} )$ & $g$ \\
\hline $X_2$ & $\neq 0, \pi$ &  0 or $\pi$ & irreducible & $|a|-1$ & $3g-2$ \\
\hline $X_3$ & 0 &  0 or $\pi$ & irred/abelian &  2 if $b$ is even  & $3g-3$ \\ 
& & & & 1 otherwise & \\
\hline $X_4$ & 1 &  0 or $\pi$ & irreducible & 1 if $b$ is odd & $3g-3$ \\
& & & & 0 otherwise & 
\\ \hline
\end{tabular}
\caption{Characteristics of the components of $\mo (S)$} \label{tab:xi}
\end{table}

For any $\rho \in \mo (S)$, we denote by $\op{CS} ( \rho ) \in \R / 2 \pi \Z$ the Chern-Simons invariant of $\rho$, cf. Equation (\ref{eq:defrhox}) for a precise definition. The Chern-Simons function $\op{CS}$ is constant on each component of $\mo (S)$. We denote again by $\op{CS}$ the induced map from $X$ to $\R / 2 \pi \Z$

For any $t \in [0,1]$, let $\mo ( \Si, t)$ denote the space of conjugacy classes of group morphism $\rho$ from $\pi_1 (\Si)$  to $\op{SU}(2)$ such that $\op{tr} ( \rho (C) ) = \cos ( 2 \pi t)$. This space is a symplectic manifold. We denote by $v_g(t)$ its symplectic volume. 
  
 
\begin{theo} \label{theo:main-result}
We have the full asymptotic expansion
\begin{xalignat*}{2}
 Z_k ( S) = & \sum_{x \in X} e^{ik \op{CS} (x) } k^{n(x)} \sum_{\ell \in \N} k^{-\ell} a_{\ell} (x) + \sum_{x \in X_3}   k^{m(x)} \sum_{\ell \in \N} k^{-\ell} b_{\ell} (x) + 
 \bigo ( k^{-\infty})
\end{xalignat*}
where the coefficients $a_{\ell} (x)$, $b_{\ell} (x)$  are complex numbers. The exponents $n(x)$ and the leading coefficients $a_0 (x)$ are given in table \ref{tab:coef} according to whether $x$ belong to $X_1$, $X_2$, $X_3$ or $X_4$. Furthermore, if $x \in X_3$ 
$$ m(x) = 2g - \tfrac{3}{2}, \qquad  b_0 (x)  = e^{i \pi/4} i^g b^{g - 3/2} a ^{-g}  \pi^{-g +1} \frac{\sqrt 2 ( g-1)!}{ ( 2 (  g-1))!} .$$ 
\end{theo}

\begin{table} 
\[
\renewcommand\arraystretch{1.7}
\begin{array}{|l|c|c|}  
\hline 
   & n(x) & a_0 (x)  \\
\hline
X_1 & g -1/2 &    \frac{2 \pi^{g - 1/2} }{ \sqrt{b}}  \bigl[ \sin ( \al (x) ) \bigr] ^{-2g + 1} \sin \bigl( c \al (x)  \bigr) \\
\hline
X_2 &   3g -2 & \frac{1 }{  \sqrt{a}}  v_{g} \bigr( \be (x) /\pi \bigl) \sin \bigl( d \be(x)   \bigr)  \\ 
\hline
X_3  & 3g -3 &  i(4 \pi)^{-1}  a^{-3/2}   v_g'(0) \\
\hline 
X_4  & 3g -3 &   (4 \pi)^{-1} a^{-3/2} v'_{g} (1) \\
\hline 
\end{array}
\]
\caption{The leading exponents and coefficients} \label{tab:coef}
\end{table}

The novelty in this result is the expression of the leading coefficients $a_{0} (x)$ and $b_0 (x)$. In the paper \cite{LL}, it is proved that $a_0 (x)$ is actually equal to the integral of a Reidemeister torsion, in agreement with the Witten asymptotic conjecture.  
Comparing Tables \ref{tab:xi} and \ref{tab:coef}, we also see that the exponent $n(x)$ is equal to half the dimension for an irreducible component, and half the dimension minus $1/2$ for the abelian components, in agreement with the prediction in \cite{FrGo}. 

All the components of $\mo (S)$ are smooth manifolds except the ones in $X_3$ which have a singular stratum of abelian representations. For these singular components, we have an additional term in the asymptotic expansion, the series $\sum k^{-\ell} b_{\ell} (x)$.  We don't know any geometric expression for $b_0 (x)$ or $m_0(x)$. For instance, since the singular strata have the same dimension as the abelian component in $X_1$, we could expect that $m(x)$ is equal to $ n(y)$, $y \in X_1$; but it is not the case. 
As a last remark, we haven't tried to determine the signs of the coefficients and to understand the spectral flow contribution.

\subsubsection*{Comparison with earlier results} 
The main reference on the subject is certainly the article \cite{Ro2} by Rozansky, where the author first shows that $Z_k (S)$ has an asymptotic expansion, the contribution of irreducible representations being presented in a residue form. In a second part, by comparing various expansions of path integrals, these residues are formally identified with Riemann-Roch number of moduli spaces, cf. Conjecture 5.1 and Proposition 5.3 of \cite{Ro2}. The connection with Theorem \ref{theo:main-result} is that these Riemann-Roch number are approximated at first order by the symplectic volumes $v_g(t)$. 

To compare with the previous articles on Witten asymptotic conjecture for Seifert manifolds, our proof has the advantage that the set $X = \pi_0 ( \mo (S))$ and the various invariants appear naturally. To the contrary, in most papers on the subject, it is first proved that $Z_k (S)$ is a sum of oscillatory terms with an explicit computation of the phases. Independently, one determines the set $X$ and the corresponding Chern-Simons invariants. After that, a one-to-one correspondence between the components of $\mo (S)$ and the various terms of the asymptotic expansion is established, such that the Chern-Simons invariant of a component is equal to the corresponding phase. 

The identification between the amplitudes and the Reidemeister torsion has been done in  a few cases similarly by comparing the results of two independent computations. One exception is the paper \cite{AnHi}, where everything is computed intrinsically, but the Seifert manifolds covered in \cite{AnHi} have all a vanishing Euler number, so they form a family disjoint to the Seifert manifolds we consider. 

The strategy we follow is inspired by our previous works \cite{oim_MCG} and \cite{LJ1}, \cite{LJ2} in collaboration with J. March\'e , where we  proved some generalized Witten conjecture for some manifolds with non empty boundary. In \cite{LJ1}, \cite{LJ2}, we considered the complement of the figure eight knot and our main tool was some q-difference relations. The paper \cite{oim_MCG} was devoted some mapping cylinders of pseudo Anosov diffeomorphisms and our main tool was the Hitchin connection. In the present paper, we prove a generalized Witten conjecture for the manifold $\Si \times S^1$. The main ingredients we use are the Verlinde formula and the Riemann-Roch theorem.

\subsubsection*{Sketch of the proof} 
The Seifert manifold $S$ being obtained by gluing a solid torus ${T}$ to $\Si \times S^1$, $Z_k (S)$ is the scalar product of two vectors  $Z_k ({T})$ and $Z_{k} (\Si \times S^1)$ of the Hermitian vector space $V_k (\partial \Si \times S^1)$. This vector space has a canonical orthonormal basis $(e_\ell, \; \ell = 1, \ldots, k-1)$. By \cite{FrGo}, the coefficients of $Z_{k} ( \Si \times S^1)$ in this basis are the Verlinde numbers $N^{g,k}_\ell$. 

These numbers can be computed in several ways. First, $N^{g,k}_\ell$ is the number of admissible colorings of a pants decomposition of $\Si$. However this elementary way is not very useful to study the large $k$ limit. At least, we learn that $N^{g,k}_{\ell}$ vanishes when $\ell$ is even. Second, the $N^{g,k}_{\ell}$ are Riemann-Roch numbers associated to the symplectic manifolds $\mo ( \Si, s)$ introduced above. This implies that for any odd $\ell$ satisfying $ 1<\ell <k-1$, we have  
\begin{gather} \label{eq:N_RR} 
 N^{g,k}_{\ell} = \Bigl( \frac{k}{2 \pi } \Bigr)^{3g -2} \sum_{n=0}^{3g-2} k^{-n} Q_{g,n} \Bigl( \frac { \ell}{k} \Bigr) 
\end{gather}
where the $Q_{g,n}$ are smooth function on $]0,1[$, $Q_{g,0} (t)$ being the symplectic volume $v_g (t)$.   
Third, by Verlinde formula, we have 
\begin{gather} \label{eq:N_Verlinde}
  N_{\ell}^{g,k} = \sum _{ m =1}^{k-1} S_{m,1} ^{1 - 2g} S_{m, \ell} \quad \text{ where } \quad S_{m, p}= \Bigl( \frac{2}{k}\Bigr)^{1/2}  \sin \Bigl(  \frac{\pi mp}{k} \Bigr).
\end{gather}

Introduce the Hermitian space $\mathcal{H}_k := \C^{\Z / 2k \Z}$ and denote by $(\Psi_{\ell}, \ell \in \Z / 2k \Z)$ its canonical basis. The family $(\mathcal{H}_k, \; k \in \N )$ may be viewed as the quantization of a two dimensional torus $M = \R^2 / \Z^2 \ni (x,y)$. This has the meaning that $(\mathcal{H}_k)$ may be identified with a space of holomorphic sections of the $k$-th power of a prequantum bundle over $M$. In this context some families $(\xi_k \in \mathcal{H}_k, \; k \in \N)$ concentrating in a precise way on a 1-dimensional submanifold of $M$ are called Lagrangian states. 

We will consider $V_k (\partial \Si \times S^1)$ as a subspace of $\mathcal{H}_k := \C^{\Z / 2k \Z}$ by setting  $e_{\ell} =  ( \Psi_{\ell} - \Psi_{-\ell}) /\sqrt 2$. We will prove that $( Z_{k} (\Si \times S^1), k \in \N)$ is a Lagrangian state supported by a Lagrangian submanifold of $M$.

To do this, we will establish and use the following characterization of Lagrangian states. 
Let $x_0, x_1 \in \R$ such that $ x_0< x_1< x_0+1$ and $\varphi$ be a function in $ \Ci (]x_0 , x_1 [, \R)$. Let $U$ be the open set $]x_0,x_1 [ \times \R/ \Z$ of $M$ and $\Ga$ be the Lagrangian submanifold $\{ (x, \varphi'(x)); \; x \in ]x_0, x_1[ \}$ of $U$. Then a family $( \xi_k = \sum \xi_k ( \ell) \Psi_\ell, k \in \N)$ is a Lagrangian state over $U$ supported by $\Ga$ if and only if 
$$ \xi_k ( 2 k x  ) = k^{-1/2 + N } e^{ 4 i \pi k \varphi (x) } \sum_{n = 0 } ^{ \infty} k^{-n } f_n ( x) + \bigo ( k^{-\infty}) $$
where  the $f_n$ are smooth functions on $]x_0, x_1[$. This formula may be viewed as a discrete analogue of the WKB expression, the function $\varphi$ is a generating function of $\Ga$, the leading order term $f_0$ of the amplitude determines the symbol of the Lagrangian state. 

By this characterization, we deduce from (\ref{eq:N_RR}) that $Z_{k} (\Si \times S^1)$ is the sum of two Lagrangian states over $M \setminus \{ x = 0 \text{ or } 1/2 \}$, supported respectively by $\{ y = 0 \}$ and $\{ y = 1/2 \}$. Indeed, we insert a factor $( 1 - (-1)^{\ell})/2$ in the right hand side of (\ref{eq:N_RR}) so that the equation is valid for even and odd $\ell$, and we use that for  $\ell = 2kx$, $(-1)^{\ell} = e^{ 4 i \pi k x/2}$.
To complete this description on a neighborhood of $\{ x = 0 \text{ or } 1/2 \}$, we will perform a discrete Fourier transform. By Verlinde formula (\ref{eq:N_Verlinde}), we easily get
\begin{gather} \label{eq:3}
 Z_{k} (\Si \times S^1) = \frac{\sqrt k }{2i  } \sum_{m \in (\Z/2k\Z)\setminus \{ 0,k \}} S_{m,1}^{1-2g} \Phi_m 
\end{gather}
where $(\Phi_m)$ is the basis of $\mathcal{H}_k$ given by $\Phi_m = (2k)^{-1/2} \sum_\ell e^{i \pi \ell m /k} \Psi_\ell$ for $m \in \Z /2k \Z$.  A similar characterization of Lagrangian state as above holds, where we exchange $x$ and $y$ and replace the coefficients in the basis $( \Psi_{\ell})$ by the ones in $( \Phi_m)$. We deduce from (\ref{eq:3}) that $Z_k ( \Si \times S^1)$ is Lagrangian on $M \setminus \{ y=0 \text{ or } 1/2 \}$ supported by $\{ x =0 \}$. 

Gathering these results, we conclude that $Z_k ( \Si \times S^1)$ is Lagrangian on the open set $ M \setminus \{ (0,1), ( 0,1/2) \}$ and supported by $\{ x= 0 \} \cup \{ y= 0\} \cup \{ y = 1/2 \}$. There is no  similar description on a neighborhood of $(0,0)$ and $(0,1/2)$ because the circle $\{x =0 \}$ intersects with $\{ y= 0 \}$ and $\{ y =1/2 \}$ at these points. 

By \cite{LJ1}, the state $Z_k ( {T})$ is Lagrangian supported by the circle $\{ y = ax/b \}$.  The scalar product of two Lagrangian states supported on Lagrangian manifolds $\Ga$, $\Ga'$ which intersects transversally, has an asymptotic expansion, and we can compute geometrically the leading order terms, each intersection point of $\Ga \cap \Ga'$ having a contribution \cite{oim_pol}. From this, we obtain Theorem \ref{theo:main-result}, except for the contribution of $X_3$ which corresponds to the singular points $(0,0)$, $(0,1/2)$.  

Actually, a large portion of this paper will be devoted to the contribution of $X_3$ in Theorem \ref{theo:main-result}.  On one hand we will establish some singular stationary phase lemma for discrete oscillatory sum. On the other hand we will prove several properties of the functions $Q_{g,n}$ in (\ref{eq:N_RR}):  first $Q_{g,n}$ vanishes for odd $n$, second $Q_{g, 2m}$ is a polynomial of degree $2(g -m) -1$, third the even part of $Q_{g, 2m}$ is a multiple of the monomial $ x^{2(g-m-1)}$ and the coefficient of this multiple with be explicitly computed for $m=0$. 

To prove these facts, we will study the discrete Fourier transform of the family $( \sin ^{-m}( \pi \ell / k ), \; \ell \in \Z/ 2k \Z)$ in the semiclassical limit $k \rightarrow \infty$. 
These families may be viewed as discrete analogue of the homogeneous distributions and are interesting by themselves. Then using Verlinde formula (\ref{eq:N_Verlinde}), we will recover the expression (\ref{eq:N_RR}) and obtain the above properties of the $Q_{g,n}$. Some of these properties also have symplectic proofs. For instance the fact that the $Q_{g,n}$ are polynomials is a consequence of Duistermaat-Heckman theorem by introducing some extended moduli space as in \cite{Je2}. The fact that the $Q_{g,n}$ vanish for odd $n$ may be deduced from Riemann Roch theorem by computing the characteristic class of the $\mo ( \Si, s)$, and expressing the Riemann Roch number in terms of $A$-genus instead of the Todd class. 

To finish this overview of the proof of Theorem \ref{theo:main-result}, let us briefly explain the topological interpretation of the previous computation.  For any topological space $T$, let $\mo ( T)$ be the space of conjugacy classes of group morphism $\pi_1 (T) \rightarrow \op{SU}(2)$. We have a natural identification between $\mo ( \partial \Si  \times S^1)$ and the quotient $N$ of $M = \R^2 /\Z^2$ by the involution $(x,y) \rightarrow (-x, -y)$. This identification restricts to a bijection between: 
\begin{enumerate} 
\item the projection of $\{x = 0\} \cup \{y =0\} \cup \{ y = 1/2 \}$ and the image of the restriction map $r: \mo ( \Si \times S^1 ) \rightarrow \mo  (  \partial \Si  \times S^1)$
\item the projection of   $\{ y = ax/b \}$ and the image of  the restriction map $r':\mo  ( {T})  \rightarrow \mo  (  \partial \Si  \times S^1)$. 
\end{enumerate}
Finally $\mo (S)$ may be viewed as the fiber product of $r$ and $r'$, its connected components being the fibers of the projection $\mo (S) \rightarrow \mo  (  \partial \Si  \times S^1)$. 
\begin{figure}[!ht]
\centering
\def\svgwidth{9cm}
  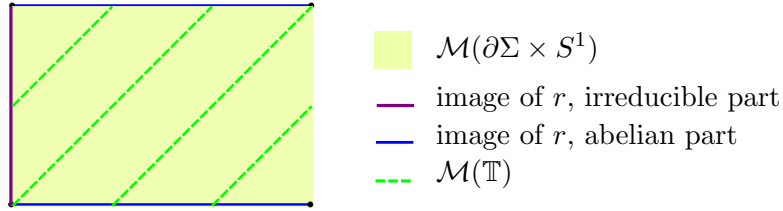 
\caption{Character variety intersections} \label{fig:dessin}
\end{figure}
So the different part of the asymptotic expansion of $\langle Z_k ( {T}), Z_k (\Si \times S^1) \rangle$ are naturally indexed by $X= \pi_0 ( \mo (S))$. In the decomposition $X = X_1 \cup X_2 \cup X_3 \cup X_4$ used in Theorem \ref{theo:main-result}, $X_1$ correspond to the points of $M$ such that $x=0$ and $y \neq 0,1/2$, $X_2$ to $ y =0$ or $1/2$ and $x \neq 0,1/2$, $X_3$ to $(0,0)$, $(0,1/2)$ and $X_4$ to $(1/2,0)$, $(1/2,1/2)$. As we will see, the Chern-Simons invariants appears naturally by interpreting the prequantum bundle of $L$ as the Chern-Simons bundle. Furthermore the expressions for the coefficients $a_0$ come from the leading order term $Q_{g,0} (t) = v_g (t)$ in (\ref{eq:N_RR}), the $(1-2g)$-th power of the sinus in (\ref{eq:N_Verlinde}) and the coefficients $a,b$ of the surgery.
 
\subsubsection*{Outline of the paper} 

In section \ref{sec:witt-resh-tura}, we recall how we compute the WRT invariants of a Seifert manifold in terms of the Verlinde numbers. In section \ref{sec:discr-four-transf}, we study the discrete Fourier transform of the negative power of a sinus and we apply this to the Verlinde numbers. In section \ref{sec:geom-quant-tori}, we recall some basic facts on the quantization of tori and their Lagrangian states. Furthermore we establish a criterion characterizing the Lagrangian states in terms of a generating function of the associated Lagrangian manifold. In section \ref{sec:asympt-behav-z_k}, we prove that $Z_k (\Si \times S^1) $ is a Lagrangian state and deduce the asymptotic behavior of $Z_k (S)$ by using some singular stationary phase lemma proved in the Section \ref{sec:sing-discr-stat}. Finally, Section \ref{sec:geom-interpr-lead} is devoted to the geometric interpretation of the results.


\section{The Witten-Reshetikhin-Turaev invariant of a Seifert manifold} \label{sec:witt-resh-tura}

For any integer $k \geqslant 2$ and any closed oriented 3-manifold $M$, we denote by $Z_k (M)$ the Witten-Reshetikhin-Turaev (WRT) invariant of $M$ for the group $\op{SU}(2)$ at level $k-2$. We are interested in the large level limit, $k \rightarrow \infty$. Since we haven't chosen a spin structure on $M$, the sequence $Z_k (M)$, $k \geqslant 2$ is only defined up to multiplication by $\tau_k^n$ where $\tau_k = e^{ \frac{3i \pi}{4} - \frac{3i \pi } { 2k}}$.  

In the case $M$ is a Seifert manifold, it is easy to compute $Z_k(M)$ by using a surgery presentation of $M$, cf. Section 1 of \cite{FrGo}. Let us explain this.  
 Let $\Si$ be a compact oriented surface with boundary a circle $C$. Let $D$ be a disc and $S^1$ be the standard circle. Consider the Seifert manifold $S$ obtained by gluing the solid torus  $ D \times S^{1}$ to $\Si \times S^1$ along a preserving orientation diffeomorphism $\varphi : \partial D \times S^1 \rightarrow C \times S^1$
$$ S = ( \Si \times S^1 ) \cup_{\varphi} ( D \times S^1)^{-} .$$
The WRT invariant of  a manifold obtained by gluing two manifolds along their boundary may be computed as a scalar product in the setting of topological quantum field theory, \cite{Wi} \cite{ReTu}.   In our particular case, we have 
\begin{gather} \label{eq:scalar_product}
 Z_k (S) = \bigl\langle Z_k ( \Si \times S^1) , \rho_k ( \varphi) ( Z_k ( D \times S^1)) \bigr\rangle_{ V_k ( C \times S^1)}  
\end{gather}
Here $V_{k} ( C  \times S^1)$ is the Hermitian vector space associated to the torus $C \times S^1$. It has dimension $k-1$. To any oriented basis $( \mu, \la)$ of $H_1 (  C \times S^1)$ is associated an orthornormal basis of $V_k ( C \times S^1)$, called the Verlinde basis. Let us choose $\mu = [ C]$, $\la = [S^1]$ and denote by $e_ \ell$, $\ell = 1, \ldots , k-1$ the corresponding basis.   

The bracket in Equation (\ref{eq:scalar_product}) denote the scalar product of $V_k ( C \times S_1)$. Furthermore for any compact oriented 3-manifold $M$ with boundary $C \times S^1$, we denote by $Z_k (M) \in V_k ( C \times S^1)$ the corresponding vector defined in the Chern-Simons topological quantum field theory.

\begin{lem} \label{lem:SiS}
One has $Z_k ( D \times S^1) = e_1$ and $$Z_k ( \Si \times S^1) = \sum_{\ell =1 }^{ k-1} N _\ell ^{g,k} e_{\ell}, $$ where $g$ is the genus of $\Si$ and $N_{\ell} ^{g,k}$ is the dimension of the vector space associated in Chern-Simons quantum field theory  to any genus $g$  surface equipped with one marked point colored by $\ell$. 
\end{lem}

In our convention, the set of colors is $\mathcal{C}_k = \{ 1, 2, \ldots, k-1 \}$, the color $\ell$ corresponding to the $\ell$-dimensional irreducible representation of $\op{SU}(2)$. 
\begin{proof} Recall first that the Verlinde basis consists of the vectors given by  
$$ e_\ell = Z_k ( D \times S^1, x, \ell) $$
where $x$ is the banded link $[0,1/2] \times S^1$ of $D \times S^1$. Since $\ell =1$ is  the trivial color in our convention,  $Z_k ( D \times S^1) = e_1$. Let us compute the coefficients of $Z_k ( \Si \times S^1)$. We have 
$$ \langle Z_{k} ( \Si \times S^1) , e_{\ell} \rangle = Z_k ( \overline{\Si} \times S^1, x, \ell) $$
where $\overline \Si$ is the closed surface $  \Si \cup_{C} D^-$. Viewing $\overline \Si \times S^1$ as the gluing of $ \overline{\Si} \times [0,1]$ with itself, we obtain that  $Z_k ( \overline{\Si} \times S^1, x, \ell) $ is the trace of the identity of $V_k ( \overline{ \Si}, \ell)$.
\end{proof}

There are several ways to compute the numbers $N _\ell ^{g,k}$. First $N_{\ell}^{g,k}$ is the number of admissible colorings of any pants decomposition of $\Si$. But this elementary way is not very useful to study the large $k$ limit. Alternatively we can use the Verlinde formula. 

\begin{theo} \label{theo:Verlinde}
For any $k \in \N^*$, $\ell = 1 , \ldots , k-1$ and $g \in \N^*$, we have 
$$ N_{\ell}^{g,k} = \sum _{ m =1}^{k-1} S_{m,1} ^{1 - 2g} S_{m, \ell}
$$
where $S_{m, p}= \bigl( \frac{2}{k}\bigr)^{1/2}  \sin ( \pi \frac{mp}{k} )$.
\end{theo} 

Later we will also use the fact that $N_{\ell}^{g,k}$ can be computed with the Riemann-Roch theorem, cf. Theorem \ref{theo:RR}.

It remains to explain the meaning of the $\rho_k$ appearing in Equation (\ref{eq:scalar_product}). $\rho_k $ is the representation of the mapping class group of $C \times S^1$ in $V_k ( C \times S^1)$ provided by the topological quantum field theory. It is actually a projective representation. More precisely, $\rho_k ( \varphi)$ is well-defined up to a power of the constant $\tau_k$ defined above. Using the basis $( \mu,\la)$, the mapping class group of $C \times S^1$ is identified with $\op{SL} ( 2, \Z)$. Then we have 
\begin{xalignat*}{3} 
\rho_k(T)e_\ell=  &  e^{\frac{i\pi(\ell^2-1)}{2k}}e_\ell,  & &  \text{ if } \quad T=\begin{pmatrix} 1&1\\ 0&1\end{pmatrix} 
\intertext{ and } 
  \rho_k(S)e_\ell=  & \sqrt{\frac{2}{k}}\sum_{\ell'=1}^{k-1}\sin \Bigl( \frac{\pi \ell\ell'}{k} \Bigr) e_{\ell'},   & & \text{ if }  S=\begin{pmatrix} 0&-1\\ 1&0\end{pmatrix}. 
\end{xalignat*}
Since $\op{SL} ( 2, \Z)$ is generated by the matrices $S$ and $T$, the representation is completely determined by these formulas.

\section{The discrete Fourier transform of a negative power of sinus} \label{sec:discr-four-transf}
 
Let $\T = \R / 2 \Z$ and $R_k = (\frac{1}{k} \Z) / 2 \Z \subset \T$. 
Introduce for any $m \in \N$, the function $\Xi_m$ from $R_k$ to $\C$ given by 
\begin{gather} \label{eq:defXi}
\Xi_{m,k} (x) =  (  i )^{-m} \sum_{ y \in R_k \setminus \{ 0, 1\} }  \bigl[ \sin ( \pi y ) \bigr]^{-m} e^{ik \pi y x }, \qquad \forall x \in R_k .
\end{gather}
This function may be viewed as the discrete Fourier transform of $y \rightarrow \bigl[ \sin ( \pi y ) \bigr]^{-m}$. We are interested in its behavior as $k$ tends to infinity.

\begin{theo} \label{theo:dev_part}
 For any $m\in \N^*$, there exists a polynomial function $P_{m}  \in \Q [k,x]$   such that 
for any $k \in \N^*$ and $x \in [0,2 ] \cap \frac{1}{k} \Z $, we have
$$ \Xi_{m,k} (x) = \bigl( 1 + ( -1 ) ^{kx + m }\bigr)  P_{m} ( k,x) .$$
Furthermore, $P_{m}$ is a linear combination of the monomials $ k^q x^p$ where $p,q$ run over the integers satisfying $0 \leqslant p \leqslant q \leqslant m$.
\end{theo}

The proof, given in Section \ref{sec:proof-theor-refth}, allows to compute inductively the polynomials $P_m$. In particular, we have 
\begin{xalignat*}{1}   & P_1 (k,x ) =  k ( - x +1 ),  \\ & P_2 (k,x)  =  k^2 \bigl( -\tfrac{1}{2} x^2 +x - \tfrac{1}{3} \bigr) + \tfrac{1}{3}, \\
&  P_3 (k,x  ) = k^3  \bigl( - \tfrac{1}{6} x^3 + \tfrac{1}{2} x^2 -\tfrac{1}{3} x \bigr) + k \bigl( \tfrac{1}{2} x -\tfrac{1}{2} \bigr).
\end{xalignat*} 
In Section \ref{sec:furth-prop-polyn}, we will establish further properties of the polynomials $P_m$. First, each $P_m$ is a linear combination of monomials $k^q x^p$ where $0 \leqslant p \leqslant q \leqslant m$ and $q \equiv m$ modulo 2. Second we will describe the singularity of $\Xi_{m,k}$ at $x =0$ as follows. Write $P_m = P_m^{+} + P_m^{-}$ where $P_m^+$ (resp. $P_m ^{-}$) is a linear combination of the $k^q x^{2 \ell}$ (resp. $k^q x^{2 \ell+1}$). Then 
$$  \frac{1}{2} \bigl( P_m  ( k ,x ) - P_m ( k, 2 + x ) \bigr) = \begin{cases} P^-_m ( k, x) \quad \text{ if $m$ is even,} \\ P^+_m ( k, x) \quad \text{ if $m$ is odd.}
\end{cases} $$
Furthermore, if $m \geqslant 1$ is even (resp. odd), $P_m^-$ (resp. $P^+_m$) is a linear combination of the monomials $k^{m-2 \ell} x^{m- 2\ell -1}$, with $\ell = 0 ,1 \ldots$. The coefficient of $k^{m} x^{m-1}$ is $ 1 /  (m-1)!$. In Section \ref{sec:appl-count-funct}, we will apply this to the counting function $N^{g,k}_{\ell}$.
 
\subsection{Proof of Theorem \ref{theo:dev_part}} \label{sec:proof-theor-refth}

In the sequel, everything depends on $k$. Nevertheless, to reduce the amount of notation, the subscript $k$ will often be omitted. Introduce the space $\mathcal{H} := \C^{R_k}$ and its scalar product 
$$ \langle f, g \rangle = \frac{1}{2k} \sum_{x \in R_k } f(x) \overline{g(x)} , \qquad f, g \in \mathcal{H} .$$ 
\subsubsection*{The endomorphisms $T$, $L$ and $\Delta$} 
Introduce the endomorphisms $T$, $L$ and $\Delta$ of $ \mathcal{H}$ given by  
$$ (Tf ) (x) = ( -1)^{kx} f(x), \qquad ( L f) ( x) := f \Bigl( x + \frac{1}{k} \Bigr), \qquad \De = \tfrac{1}{2} \bigl( L - L^{-1} \bigr)$$ 
Observe that $T$ is unitary, $T^2 = \op{id}$ and $\mathcal{H} = \mathcal{H}^+ \oplus^{\perp} \mathcal{H}^-$  where $\mathcal{H}^{\pm} = \ker ( T \mp \op{id}).$ Furthermore, $\mathcal{H}^+$ (resp. $\mathcal{H}^-$) consists of the functions vanishing on the $x \in R_k$ such that $kx$ is odd (resp. even). The orthogonal projector of $\mathcal{H}$ onto $\mathcal{H}^{\pm}$ is $\frac{1}{2} ( \op{id} \pm T)$.  

Let $u_0 \in \mathcal{H} $ be the function constant equal to $1$, $u_1 = T u_0$ and $u_0 ^\pm   = \frac{1}{2} ( u_0  \pm  u_1 ) $. Denote by $\mathcal{H}_0^{\pm}$ the subspace of $\mathcal{H}^\pm$ orthogonal to $u_0 ^\pm$.  
  
\begin{lem} The kernel of $\Delta$ is  spanned by $u_0$ and $u_1$. Furthermore $\Delta$ restricts to a bijection from $\mathcal{H}_0^{\pm} $ to $  \mathcal{H}_0^{\mp}$.
\end{lem}

\begin{proof} We easily check that $u_0$ and $u_1$ belong to the kernel of $\Delta$ and that this kernel is 2-dimensional. 
Furthermore we have that $LT + T L = 0$, so that $\Delta T + T \Delta = 0$ and consequently $ \Delta ( \mathcal{H}^{\pm} ) \subset \mathcal{ H}^{\mp}$. Since $L$ is unitary, $\Delta $ is skew-Hermitian. So the range of $\Delta$ is the subspace of $\mathcal{H}$ orthogonal to $u_0$ and $u_1$.  
\end{proof} 
\subsubsection*{Discrete Fourier transform} 
Let $(u_y, \; y \in R_k )$ be the orthonormal basis of $\mathcal{H}$ 
$$ u_{y} (x) =  e^{ik\pi y  x }, \quad \forall x \in R_k  $$ 
When $y=0$ or $1$, we recover the functions $u_0$ and $u_1$ introduced previously. 
Denote by $\mathcal{F} : \mathcal{H} \rightarrow \mathcal{H}$ the discrete Fourier transform
$$ \mathcal{F} ( f) ( y) = \langle f, u_y \rangle , \qquad \forall y \in R_k . $$
Using the relations $ T u _y = u_{y+1}$ and $ L u _y = e^{i \pi y} u_y$, we deduce that for any $f \in \mathcal{H}$ and $y \in R_k$ 
\begin{gather} \label{eq:fourierTDelta}
 \mathcal{F} (T f )  ( y) = \mathcal{F}(f) ( y- 1), \qquad \mathcal{F} ( \Delta  f)  ( y) =  i \sin ( \pi y ) \mathcal{F} ( f) ( y)  . 
\end{gather} 
By definition of $\Xi_m$, Equation (\ref{eq:defXi}), its discrete Fourier transform is given by 
$$  \mathcal{F} ( \Xi_m ) ( y)  = \begin{cases}   \bigl[ i\; \sin ( \pi y) \bigr]^{-m} \text{ if } y \neq 0, 1 \\ 0 \qquad \text{ otherwise.} \end{cases}
$$
Observe that  $\mathcal{F} ( \Xi_m ) ( y -1 ) = (-1)^m \mathcal{F} ( \Xi_m ) ( y)$,  so that $T \Xi_m = ( -1)^m \Xi_m$. Furthermore $\mathcal{F} ( \Xi_m ) ( 0) = \mathcal{F} ( \Xi_m ) ( 1) = 0$ implies that $\Xi_m$ is orthogonal to both $u_0$ and $u_1$. Hence, $\Xi_m$ belongs to $\mathcal{H}^{+}_0$ or $\mathcal{H}^{-}_0$ according to whether $m$ is even or odd. Furthermore we have that $\Delta \Xi_{m+1} = \Xi_m$. So we can inductively compute $\Xi_m$ by inverting the operators $\Delta: \mathcal{H}^{\pm}_0 \rightarrow \mathcal{H}^{\mp}_0$. 

Let us first compute $\Xi_0$. We have 
$$ \Xi_0  =  \Bigl( \textstyle{\sum}_{y \in R_k} u_y \Bigr) -  ( u_0 + u_1 ) .$$
Furthermore $\sum_{y \in R_k} u_y  = 2k \delta$ where $\delta ( x) = 1$ if $x = 0$ and $\delta ( x) = 0$ otherwise. So we obtain
\begin{gather} \label{eq:xi0}
 \Xi_0 (x) = 2 k \delta (x) - 1 - ( -1)^{kx} .
\end{gather}  

Let us compute the inverse of $\Delta :\mathcal{H}^-_0  \rightarrow \mathcal{H}^+_0 $.  Denote by $S_k $ the set of integers $\{1, \ldots, k \}$. Let $\tilde{L} $ and $L$ be the endomorphisms of $\C^{S_k}$ defined by 
$$  \quad \forall \ell \in S_k, \quad \tilde{L} (f) ( \ell ) = \sum _{m = 1}^{ \ell} f (m), \qquad L ( f) = 2 \tilde{L} (f) - \frac{2}{k} \sum_{1}^{k} \tilde{L} (f)(\ell)  $$
Let us identify $\mathcal{H}^+$ and $\mathcal{H}^-$ with $\C^{S_k}$ by sending $g^{\pm} \in \mathcal{H}^\pm$ into $f^{\pm} \in \C^{S_k}$ given by 
\begin{gather}\label{eq:rel1}
 f^{+} ( \ell ) = g^+\Bigl( \frac{2(\ell-1)}{k } \Bigr) , \qquad f^{-} ( \ell ) = g^{-} \Bigl(  \frac{2 \ell -1 } { k}  \Bigr) .
\end{gather} 
Observe that the subspaces $ \mathcal{H}^{+}_0$ and $\mathcal{H}^{-}_0$ get identified with the subspace of $\C^{S_k}$ consisting of function with vanishing average. 
\begin{lem}
The inverse of $\Delta : \mathcal{H}^-_0  \rightarrow \mathcal{H}^+_0$ is $L$.  
\end{lem}
\begin{proof} 
Let $\tilde{\Delta}$ be the endomorphism of $\C^{S_k}$ corresponding to $\Delta : \mathcal{H} ^- \rightarrow \mathcal{H} ^+$ through the identifications $\mathcal{H}^- \simeq \C^{S_k}$ and $\mathcal{H}^+ \simeq \C^{S_k}$. A straightforward computation shows that for any $f \in \C^{S_k}$, we have $$\tilde{\Delta} f  (\ell) = \frac{1}{2} ( f ( \ell) - f ( \ell -1 )) , \qquad \ell \in S_k$$ with the convention that $f (0) = f(k)$.  Assume that the average of $f$ vanishes, so that $\tilde {L} f ( k) = \tilde{L} f ( 0)$. So we have that $2 \tilde{\Delta} \tilde{L} f = f$.  Since $\tilde{\Delta} 1 = 0 $, it follows that $ \tilde{\Delta} L f = f$. Furthermore $Lf$ has a vanishing average. 
\end{proof}

Let us apply this to compute $\Xi_1$. By (\ref{eq:xi0}) and (\ref{eq:rel1}), the function $f^{+}_0 \in \C^{S_k}$  corresponding to $\Xi_0$ is given by $f^{+}_0 ( 1) = 2k -2 $ and $f^{+}_0 ( \ell ) = -2$ for $\ell = 2, \ldots, k$. So that $\tilde{L} ( f^{+}_0)(\ell)  = 2k - 2\ell$ and $L( f^{+}_0)( \ell) = -4 \ell + 2k + 2 $. Inverting the second relation of (\ref{eq:rel1}), we obtain for any $x$ of the form $(2 \ell + 1)/k$,  
\begin{gather} \label{eq:rel2}
 g^{-} ( x) = f^{-} \Bigl( \frac{kx + 1 }{2}  \Bigr) 
\end{gather}  
 which leads to $ \Xi_1 ( x )   = 2 k ( - x +1 )$. This proves Theorem \ref{theo:dev_part} for $m=1$. 

Assume now that Theorem \ref{theo:dev_part} has been proved for some even $m \geqslant 2$. So the restriction of $\Xi_m $ to $[0,2] \cap \frac{2}{k} \Z $ is a linear combination of the monomials $ k^ q x^p$ where $ p \leqslant q \leqslant m$. Then, the function $f_m^+ \in \C^{S_k}$ corresponding to  $\Xi_m $ is a linear combination of the monomials $ k^ {q-p} \ell^p$ where $ p \leqslant q \leqslant m$. Now, it is a well-known consequence of Euler-Maclaurin formula that for any $p  \in \N$, there exists a polynomial $Q_p$ with degree $p+1$ and vanishing at $0$ such that 
\begin{gather} \label{eq:sompol} 
 \sum_{m = 1 } ^{\ell} m ^p = Q_p ( \ell ) , \qquad \forall \ell \in \N  
\end{gather}
Let $f \in \C^{S_k}$ be given by $f( \ell) = \ell^p$. Then $\tilde{L}( f) ( \ell) = Q_p ( \ell)$. 
Applying (\ref{eq:sompol}) to the monomials of $Q_p$, we obtain a polynomial $R_p$ of degree $p+ 2$, vanishing at $0$ and such that  
$$ \sum_{m = 1 } ^{\ell} Q_p ( m) = R_{p } ( \ell) , \qquad \forall \ell \in \N   
$$
Consequently $ L ( f) ( \ell ) = 2Q_p ( \ell) - 2k^{-1}R_{p } (k) $, so that 
$$ L( k^{q-p} f ) ( p) = 2k^{q-p} Q_p ( \ell ) -  2k^{q-p-1} R_p ( k) . $$ 
Since $R_p(0) = 0$, $k^{-1} R_p ( k)$ is polynomial in $k$ with degree $p+1$.  
This proves that $L ( f_m)$ is a linear combination of the monomials $ k^ {q-p} \ell^p$ where $ p \leqslant q \leqslant m +1 $. Applying the relation (\ref{eq:rel2}), we obtain that the restriction of $\Xi_{m+2} $ to $[0,2] \cap \frac{1}{k} ( 1+ 2\Z)$ is  a linear combination of the monomials $ k^ {q-p + p'} x^{p'}$ where $ p' \leqslant p \leqslant q \leqslant m +1 $. Equivalently it is a linear combination of the $ k^ {q} x^{p} $ where $ p \leqslant q \leqslant m +1 $, which proves Theorem \ref{theo:dev_part} for $m+1$. 

Similarly we can show that the result for $m$ odd implies the result for $m+1$. To do this, we identify $\mathcal {H}_0^+$ and $\mathcal{H}_0^-$ with $\C ^{S_k}$ by using the relations
 \begin{gather*}
 f^{+} ( \ell ) = g^+\Bigl( \frac{2\ell}{k } \Bigr) , \qquad f^{-} ( \ell ) = g^{-} \Bigl(  \frac{2 \ell -1 } { k}  \Bigr) .
\end{gather*} 
instead of (\ref{eq:rel1}). Then the inverse of $\Delta : \mathcal{H}^+_0  \rightarrow \mathcal{H}^-_0$ is still given by $L$. The remainder of the proof is unchanged. 

\subsection{Further properties of the polynomials $P_{m}$} \label{sec:furth-prop-polyn}

For any $m$, let us write $P_m = P_m^{+} + P_m ^{-}$, where $P_m^{+}$ (resp. $P_m ^{-}$) is a linear combination of the monomials $ k^q x^{2 \ell}$ (resp. $ k^q x^{2 \ell + 1}$). 
\begin{prop} \label{eq:sing_part}
For any $m \in \N^*$, we have 
$$ \frac{1}{2} \bigl( P_m  ( k ,x ) - P_m ( k, 2 + x ) \bigr) = \begin{cases} P^-_m ( k, x) \quad \text{ if $m$ is even,} \\ P^+_m ( k, x) \quad \text{ if $m$ is odd.} 
\end{cases}$$
\end{prop}

\begin{proof}
The discrete Fourier transform $\mathcal{F}( \Xi_m)$ has the same parity as $m$, so the same holds for $\Xi_m$. This implies that for any $x \in [0,2] \cap \frac{1}{k} \Z$ such that $kx $ has the same parity as $m$, we have 
$$ P_m( k, 2- x ) = ( -1)^m P_m( k, x)$$
Since $P_m$ is polynomial in $k$ and $x$, this equality actually holds for any $k$ and $x$. So we have
$$ P_m ( k, x ) - P_m ( k, 2 + x ) = P_m ( k,x ) - ( -1)^m P_m (k,  -x) ,
$$
which concludes the proof.
\end{proof}

\begin{prop} \label{prop:calcul_pm}
For any $m \geqslant 2$, we have 
\begin{gather} \label{eq:rec_Pm} 
 \sum_{\ell =1 }^{\infty} \frac{k^{-2 \ell +1 }}{(2 \ell -1)!} \Bigl( \frac{d}{dx} \Bigr)^{2 \ell - 1 } P_{m} ( k, x) =  P_{m-1} ( k, x).
\end{gather}
Furthermore, if $m$ is even, 
\begin{gather} \label{eq:pair_init}
 k \int_0^2 P_m ( k, x ) \; dx =  4 \sum_{\ell=1}^{\infty} \frac{B_{2 \ell} }{ ( 2 \ell ) !} \Bigl( \frac{2}{k} \Bigr)^{ 2 \ell -1 } \Bigl( \frac{d}{dx} \Bigr)^{2 \ell - 1 } P_m^- ( k,0) 
\end{gather}
where the $B_{\ell}$ are the Bernouilli numbers. If $m$ is odd,  
\begin{gather} \label{eq:impair_init} P_{m} (k, \tfrac{1}{k} ) =  P_{m-1} ( k,0) .
\end{gather} 
\end{prop}

\begin{proof} 
Since $\Delta \Xi_{m} = \Xi_{m-1}$, we have for any $x \in [0,2] \cap \frac{1}{k} \Z$ such that $kx$ has the same parity as $m-1$, 
$$\tfrac{1}{2} \bigl(  P_{m} (k, x + \tfrac{1}{k} ) - P_{m} (k, x - \tfrac{1}{k} ) \bigr) = P_{m-1}(k, x ) .$$
$P_{m}$ and $P_{m-1}$ being polynomials the same equality holds for any $x$ and $k$. We obtain Equation (\ref{eq:rec_Pm}) by applying Taylor formula. 

Recall that the average of $\Xi_m$ vanishes. For even $m$ this implies that 
$\sum_{\ell = 1}^{k} P_{m} ( k, \tfrac{2\ell}{k}  ) = 0 .$ By Euler-Maclaurin formula, we have
\begin{xalignat*}{2}
 \sum_{\ell = 1}^{k} P_{m} ( k, \tfrac{2\ell}{k}  ) = & \int_0 ^k P_m (k,  \tfrac{2x}{k} ) dx + \frac{1}{2} \bigl( P_m ( k, 2 ) - P_m ( k,0) \bigr) \\  + &  \sum_{ \ell \geqslant 1} \frac{B_{2\ell}}{ ( 2 \ell)!} \Bigl( \frac{2 }{k} \Bigr)^{ 2 \ell -1 }  \Bigl[ \Bigl( \frac{d}{dx} \Bigr)^{2 \ell - 1 } P_m ( k,  2 ) - \Bigl( \frac{d}{dx} \Bigr)^{2 \ell - 1 } P_m ( k,0 ) \Bigr]
\end{xalignat*} 
Using that $P_m (  k,2) = \Xi_{m, k } ( 0 ) = P_m ( k,0)$ and Proposition \ref{eq:sing_part}, we obtain Equation (\ref{eq:pair_init}). 

If $m$ is odd, $\Delta \Xi_{m} = \Xi _{m-1}$ implies that 
$$ \tfrac{1}{2} \bigl( P_m ( k, \tfrac{1}{k} ) -  P_m ( k, - \tfrac{1}{k}  ) \bigr) = P_{m-1} (k,0).$$
Since $\Xi_m $ is also odd, we have $P_m (k , - \tfrac{1}{k} ) = - P_m ( k, \tfrac{1}{k})$ which proves Equation (\ref{eq:impair_init}).
\end{proof} 

Proposition \ref{prop:calcul_pm} allows to compute the $P_m$'s inductively. Indeed, we have the following 
\begin{prop} \label{prop:recurrence}
For any $m \geqslant 2$, $P_{m}$ is the unique solution in $\C [ k,x]$ of Equations (\ref{eq:rec_Pm}) and (\ref{eq:pair_init}) (resp. (\ref{eq:impair_init})) if $m$ is even (resp. odd).
\end{prop}

\begin{proof} 
Write 
$ P_m ( k,x) = k^{r} Q_r ( x) + k^{r-1} P_{r-1} ( x) + \ldots + Q_0 (x)$. Denote by $D$ the derivation $\frac{d}{dx}$. Then 
\begin{xalignat*}{2} 
  \sum_{\ell =1 }^{\infty} \frac{k^{-2 \ell +1 }}{(2 \ell -1)!} D^{2 \ell - 1 } P_{m} ( k, \cdot ) = & k^{r-1} D Q_r  + k^{r-2} DQ_{r-1}    +   k^{r-3 }( DQ_{r-2}   + \\ &   \tfrac{1}{3!} D^3 Q_{r}  ) + k^{r-4} ( DQ_{r-3} + \tfrac{1}{3!} D^3 Q_{r-1}) + \ldots \\ = &  \sum_{\ell =0} ^{r} k^{r -\ell-1 }( D Q_{r -\ell} +  R_{\ell}) + \bigo(k^{-2} )  
\end{xalignat*}
where for any $ 0 \leqslant \ell \leqslant r$, $R_{\ell} \in \C [x]$ only depends on $Q_{r}, Q_{r-1}, \ldots , Q_{r-\ell +1}$. 
So Equation (\ref{eq:rec_Pm}) leads to a triangular system of equations for the $DQ_{r-\ell}$'s. We conclude that $P_m$ is the unique solution of Equation (\ref{eq:rec_Pm}) up to some polynomial in $k$. Arguing similarly, we prove  that this latter polynomial is uniquely determined by Equation (\ref{eq:pair_init}) or (\ref{eq:impair_init}) according to the parity of $m$.
\end{proof}

\begin{prop} \label{prop:par_k}
For any $m \in \N^*$,  $P_m $ is a linear combination of the monomials $ k^q x^p$, where $p,q$ run over the integers such that $p \leqslant q \leqslant m$ and $q$ has the same parity of $m$. 
\end{prop}
\begin{proof}
We have to prove that $P_m ( -k, x) = ( -1)^m P_m ( k, x)$. Assume the result holds for $m-1$. Then we easily see that $P_m ( -k , x)$ satisfies the Equations of Proposition \ref{prop:calcul_pm}. To check Equation (\ref{eq:impair_init}) in the case $m$ is odd, we use that $P_m ( k, -\frac{1}{k} ) = -P_m ( k, \frac{1}{k} )$. We conclude with Proposition \ref{prop:recurrence}.
\end{proof}

\begin{prop} \label{prop:sing_part}
For any even $m \geqslant 2$ (resp. odd $m \geqslant 1$), $P_m^-$ (resp. $P^+_m$) is a linear combination of the monomials $k^{m-2 \ell} x^{m- 2\ell -1}$, with $\ell = 0 ,1 \ldots$. The coefficient of $k^{m} x^{m-1}$ is $ 1 /  (m-1)!$. 
\end{prop} 

\begin{proof} Again the proof is by induction on $m$. Assume that $m$ is even and write $ A = P_m ^{-}$, $B = P_{m-1} ^{+}$. By Proposition \ref{prop:par_k}, we have
\begin{gather*} 
 A ( k, x ) = k^{m} A_m ( x) + k^{m-2} A_{m-2} ( x) + \ldots + A_0 ( x) , \\ \qquad B(k, x) = k^{m-1} B_{m-1} (x) + k^{m-3} B_{m-3} (x) + \ldots + k B_1 (x) .
\end{gather*}
We deduce from Equation (\ref{eq:rec_Pm}) that
\begin{gather*} 
 D A_m =  B_{m-1} , \qquad DA_{m-2} + \tfrac{1}{3!} D^3 A_m =  B_{m-3} , \qquad \ldots \\
D A_2 + \tfrac{1}{3!} D^3 A_4 + \ldots , \tfrac{1}{ (m-1)!} D^{m-1} A_m =  B_1 ,\\ DA_0+ \tfrac{1}{3!} D A_2 + \ldots + \tfrac{1}{ ( m +1 ) !} D^{m+1} A_m = 0 
\end{gather*} 
Assume that $ B_{\ell} (x) = b_{\ell} x^{\ell-1}$. Then these equations imply that $A_\ell ( x) = a_{\ell} x^{\ell-1} + a_{\ell}^0$. Since the $A_\ell$'s are odd, the constants $a_{\ell}^{0}$ vanish. 

Assume now that $m$ is odd and that the results holds for $m-1$. Arguing as above with Proposition \ref{prop:par_k} and Equation (\ref{eq:rec_Pm}), we prove that 
\begin{xalignat*}{2}
P_m ^+ ( k, x) =  & k^{m } ( c_m x^{m-1} + d_m ) + k^{m-2} ( c_{m-2} x^{m-3} + d_{m-2} ) + \ldots \\ + & k^{3} ( c_3 x^2 + d_3)  + k c_1
\end{xalignat*}
By Proposition \ref{prop:par_k}, $P_{m-1} ( k,0)$ is even. So Equation (\ref{eq:impair_init}) implies that $$P_m ^+ (k,  \tfrac{1}{k}  ) = 0.$$
We deduce that the $d_{\ell}$'s vanish and $c_m + c_{m-2} + \ldots + c_1 = 0$.  
\end{proof}

\subsection{Application to the counting function} \label{sec:appl-count-funct}

As a consequence of Verlinde formula, we have the following

\begin{lem} \label{lem:relat-with-count}
For any $k \in \N^*$, $\ell = 1 , \ldots , k-1$ and $g \in \N^*$, we have 
$$N^{g,k}_{\ell}  = C_g  k ^{ g -1 }  \Xi _{2g -1 } ( \ell / k ) $$
with $C_g =  (-1)^{g-1} 2 ^{-g}$.
\end{lem} 

\begin{proof}
We compute from Theorem \ref{theo:Verlinde} 
\begin{xalignat*}{2} 
N_{\ell }^{g,k}  = & \sum_{m= 1 } ^{k-1} S_{m,1} ^{1 - 2g} S_{m, \ell} =\frac{1}{2i} \Bigl( \frac{2}{k} \Bigr)^{1/2} \sum_{m=1}^{k-1} S_{m,1} ^{1 - 2g} ( e^{ i \pi \frac{m\ell}{k} } - e^{ -i \pi \frac{m\ell}{k} } )  \\
\intertext{ setting $ m = k y $ so that $S_{m, p}= \bigl( \frac{2}{k}\bigr)^{1/2}  \sin ( \pi y )$, we get} 
= & \frac{1}{2i} \Bigl( \frac{2}{k} \Bigr)^{1-g} \sum_{y \in R_k \setminus \{ 0 , 1\}}  \bigl[ \sin ( \pi y ) \bigr]^{1-2g} e^{i \pi y \ell }  \\
= & 2^{-g}(-k ) ^{g-1} i^{1-2g}  \sum_{y \in R_k \setminus \{ 0 , 1\}}  \bigl[ \sin ( \pi y ) \bigr]^{1-2g} e^{i \pi y \ell }
\end{xalignat*}
which ends the proof.
\end{proof} 

So we deduce from Theorem \ref{theo:dev_part} and the result of Section \ref{sec:furth-prop-polyn} the following 
\begin{theo} \label{theo:counting_smoot}
Let $g\in \N^*$. Then there exists a family of polynomial functions 
$P_{g, m} : [0,1] \rightarrow \R$, $m =0,1,\ldots, g-1$ such that for any $k \in \N^*$ and for any  odd integer $\ell$ satisfying $1 \leqslant \ell \leqslant k-1$, we have 
$$ N^{g, k }_{\ell}   = \Bigl( \frac{k}{2\pi} \Bigr)^{ 3 g -2 }  \sum _{m = 0 }^{g-1 } k^{-2m} P_{g, m} \Bigl( \frac{\ell}{k} \Bigr) . $$
Furthermore $P_{g, m}$ has degree $2(g -m) -1 $. The even part of $P_{g, m }$ is of the form $\la_{g,m} x^{ 2( g -m - 1)}$. For $m=0$, we have 
$$ \la_{g,0} = \frac{2 C_g ( 2 \pi)^{ 3g -2}}{ ( 2 ( g-1))!} .$$
\end{theo} 

The polynomials $P_{g,m}$ will be expressed in Section \ref{sec:symplectic-volumes} as integrals of characteristic classes on some moduli spaces. In particular $P_{g,0} (s)$ is a symplectic volume.

\section{Geometric quantization of tori and semiclassical limit} 
\label{sec:geom-quant-tori}
\subsection{The quantum spaces} \label{sec:quantum-spaces}

Let $(E,\om)$ be a real two-dimensional symplectic vector space. Let $R$ be a lattice of $E$ with volume $4\pi$. Let $L_E = E \times \C$ be the trivial Hermitian line bundle over $E$. Endow $L_E$ with the connection $d + \frac{1}{i} \al$ where $\al \in \Om^1 ( E, \C)$ is given by
$$ \al_x ( y) = \frac{1}{2} \om ( x, y). $$
Consider the Heisenberg group $E \times U(1)$ with the product 
\begin{gather} \label{eq:action_prod}
 (x, u) . ( y,v) = ( x+ y, uv e^{i \om ( x, y) /2})
\end{gather}
This groups acts on $L_E$ by preserving the metric and the connection, the action of $(x,u)$ being given by formula (\ref{eq:action_prod}) with $(y,v ) \in L_E$. The group $E \times U(1)$ is actually the group of automorphisms of $(L_E , d + \frac{1}{i} \al)$ lifting the translations of $E$. 

Since $\om ( R, R) \subset 4 \pi \Z$, $R \times \{1 \}$ is a subgroup of $E \times U(1)$. Let $M$ be the torus $E /R$ and $L_M $ be the bundle $L_E / R \times \{1\}$. The symplectic form $\om$ and the connection $d + \frac{1}{i} \al$ descend to $M$ and $L_E$ respectively. Let $k$ be a positive integer. For any $ x \in \frac{1}{2k} R$, the action of $(x,1)$ on $L_E^k$ commutes with the action of $R \times \{1 \}$. This defines an action of $(x, 1)$ on $L^k_M$. Denote by $T^*_x$ the pull-back of the sections of $L^k_M$ by the action of $(x,1)$. Observe that for any $ x, y \in R$, we have 
$$ T_{x/2k}^* T_{y/2k} ^* = e^{i\om ( x, y)/4k} T_{y/2k}^* T_{x/2k}^* .$$

Choose a linear complex structure $j$ of $E$ compatible with $\om$. 
This complex structure descends to $M$. Furthermore, $L_M$ inherits a holomorphic structure compatible with the connection. The space $H^0 ( M , L^k_M)$  of holomorphic sections of $L^k_M$ has dimension $2k$. The operators $T_{x/2k}^*$, $x \in R$ introduced above, preserve $H^0 ( M , L^k_M)$. 
 
Let $( \delta, \varphi)$ be a half-form line, that is $\delta$ is a complex line and $\varphi$ is an isomorphism from $\delta^{\otimes 2} $ to the canonical line $K_j$, 
$$ K_j = \{ \al \in E^* \otimes \C/ \al ( j \cdot ) = i \al \} .$$ 
Let $\mathcal{H}_k = H^0 ( M , L^k_M) \otimes \delta =  H^0 ( M , L^k_M \otimes \delta_M )$ where $\delta_M$ is the trivial line bundle $M \times \delta$. $K_j$ has a natural scalar product given by $( \al , \be) = i \al \wedge \con{\be} / \om$. We endow $\delta$ with the scalar product making $\varphi$ a unitary map. $\mathcal{H}_k$ has a scalar product defined by integrating the pointwise scalar product against $\om$. 

Let $( \mu , \la)$ be a positive basis of $R$, so that $\om ( \mu , \la ) = 4 \pi$. It is a known result that $\mathcal{H}_k$ has an orthonormal basis $(\Psi_{\ell} , \; \ell \in \Z / 2 k \Z)$ such that 
$$ T^*_{\mu/2k } \Psi_\ell =  e ^{ i \pi \ell / k} \Psi_\ell, \qquad T^* _{\la/ 2k } \Psi_{\ell} = \Psi_{\ell + 1} . $$ 
The only indeterminacy in the choice of this basis is the phase of $\Psi_0$. We will often use the following normalization 
$$ \Psi ( 0) = \la \si^k \otimes \Om_{\mu} \qquad \text{ with } \la >0  $$
Here $\si \in L_{M, [0]}$ is the vector send to $1$ by the identification $L_{M, [0]} \simeq L_{E, 0} = \{0 \} \times \C$. $\Om_{\mu}$ is one of the two vectors in $\delta$ satisfying $ \varphi ( \Om_{\mu}^2) ( \mu ) = 1$. We can explicitly compute the coefficient $\la$ as an evaluation of the Riemann Theta function. It satisfies as $k$ tends to infinity $\la  = 1 + \bigo ( e^{-k/C})$ for some positive $C$. 

Let $S$ be the automorphism of $L_E^k$ sending $(x,u)$ into $( -x,u)$. This automorphism descends to an automorphism of $L_M$ that we still denote by $S$. The subspace of alternating sections of $\mathcal{H}_k$ is by definition the eigenspace  $\ker ( \op{Id}_{\mathcal{H}_k} + S^k \otimes \op{id}_{\delta}) $. It has dimension $k-1$ and admits as a basis the family $(\Psi_{\ell} - \Psi_{-\ell}, \ell =1, \ldots, k)$.

\subsection{Semi-classical notions} 

Consider the same data as above. For any $k\in \N^*$ and $\xi \in \mathcal{H}_k$, we denote by $|\xi | \in \Ci ( M , \R)$ the pointwise norm of $\xi$ and by $\| \xi \| \in \R$ the norm defined previously, so $\| \xi \|^2 = \int_M | \xi |^2 \om$.  
 
We say that a family $ \xi = ( \xi_k \in \mathcal{H}_k, k \in \N^*)$ is {\em admissible} if there exists  $N$ and $C>0$ such that for any $k$, $\| \xi_k \| \leqslant C k^{N}$. Equivalently, $\xi$ is admissible if there exists $N$ and $C>0$ such that for any $k$, $| \xi_k | \leqslant C k^{N}$ on $M$. 

The {\em microsupport} of an admissible family $( \xi_k)$ is the subset $\op{MS} ( \xi )$ of $M$ defined as follows: $x \notin \op{MS} ( \xi)$ if and only if there exists a neighborhood $U$ of $x$ and a sequence $(C_N)$ such that for any integer $N $ and $x \in U$, $|\xi ( x) | \leqslant C_N k^{-N}$. 

Let $U$ be an open set of $M$. Let $\Ga$ be a one dimensional submanifold of $U$. Let $\Theta$ and $\si$ be sections of $L_M \rightarrow \Ga$ and $\delta_M  \rightarrow \Ga$ respectively. Assume that $\Theta$ is flat and that its pointwise norm is constant equal to 1. Let $\xi$ be an admissible family.  We say that the restriction of $\xi$ to $U$ is a {\em Lagrangian state} supported by $\Ga$ with associated sections $( \Theta, \si)$ if  $\op{MS} ( \xi ) \cap U \subset \Ga $ and for any $ x_0 \in \Ga$, there exists an open neighborhood $V$ of $x_0$ such that 
\begin{gather} \label{eq:lag_state}
 \xi( x)   = \Bigl( \frac{k}{2 \pi} \Bigr)^{1/4+N} E^k(x)  f( x , k) + \bigo ( k^{-\infty}) , \qquad x \in V  
\end{gather}
where the $\bigo$ is uniform on $V$ and 
\begin{itemize} 
\item $E$ is a section of $L_M \rightarrow V$ such that $ E = \Theta$ on $\Ga \cap V$, $|E| < 1$ outside of $\Ga$, $\con{\partial} E \equiv 0$ modulo a section vanishing to infinite order along $\Ga$, 
\item $(f( . , k))$ is a sequence of $\Ci ( V, \delta_M)$ which admits an asymptotic expansion of the form $f_0 + k^{-1} f_1 + \ldots$ with coefficients $f_0, f_1, \ldots $ in $\Ci (V, \delta_M)$. Furthermore $f_0 = \si$ on $\Ga \cap V$,
\item $N$ is a real number which does not depend on $x_0$. 
\end{itemize} 
Let us recall how we can estimate the norm and the scalar product of Lagrangian states in terms of the corresponding sections $\Theta$ and $\si$.  
In these statements, we identify $ \si^2 \in \Ci ( \Ga, \delta_M^2)$ with the one-form of $\Ga$ given by
$$ \si^2 (p)( X) := \varphi ( \si^2(p)  ) ( X), \qquad \forall p \in \Ga \text{ and }  X \in T_p \Ga $$
where we see $T_p \Ga$ as a subspace of $T_p M = E$. 
The normalization for $N$ has been chosen so that for any $\rho \in \Ci ( M)$ supported in $U$
$$ \int_U   | \xi _k  | ^2 \rho \; \om  = \Bigl( \frac{k}{2\pi} \Bigr)^N  \int_{\Ga}  \rho (x) |\si|^2 (x)   + \bigo ( K^{N-1}) $$
Here $| \si |^2 $ is the density $|\si^2|$ of $\Ga$, so that it makes sense to integrate it on $\Ga$.  For a proof of this formula, cf Theorem 3.2 in {\cite{oim_demi}}.

Consider now two Lagrangian states $\xi$ and $\xi'$ over $U$ with associated data $(\Ga, \Theta,\si ,N)$ and $( \Ga', \Theta', \si', N')$. Assume that $\Ga \cap \Ga ' = \{ y \}$ and this intersection is transverse. Introduce a function $\rho \in \Ci (M)$ such that $\op{supp} \rho \subset U$ and $\rho = 1 $ on a neighborhood of $y$. By Theorem 6.1 in \cite{oim_pol}, we have the following asymptotic expansion
\begin{gather} \label{eq:scalprodlag}
 \int_U    \bigl( \xi_k , \xi'_k \bigr) \rho \om  = \Bigl( \frac{k}{2\pi} \Bigr)^{-1/2 +N +N'} \bigl( \Theta (y) , \Theta' ( y) \bigr)^k_{L_{y}} \sum_{\ell = 0 }^{\infty}  k ^{\ell} a_\ell  + \bigo( k^{-\infty})  
\end{gather}
where the $a_{\ell}$'s are complex numbers and $ a_0 = \langle \si (y), \si '(y) \rangle_{T_y \Ga, T_y \Ga'} .$

Here the bracket is defined as follows. For any two distinct lines $\nu$, $\nu'$ of $E$, there exists a unique sesquilinear map 
$ \langle \cdot, \cdot \rangle_{\nu, \nu'} : \delta \times \delta \rightarrow \C$
such that for any $u$, $v \in \delta$, 
\begin{gather}  \label{eq:pairing}
  \bigl( \langle u , v \rangle_{\nu, \nu'}\bigr)^2  = i \frac{ u^2(X) \overline{v^2(Y)}}{ \om (X, Y)}, \qquad \forall X \in \nu, Y \in \nu'
\end{gather} 
where $X$ and $Y$ are any non vanishing vectors in $\nu$ and $\nu'$ respectively. The sign of $ \langle u , v \rangle_{\nu, \nu'}$  is determined by the following condition: the bracket  depends continuously on $\nu$, $\nu'$ and $\langle u, u \rangle _{\nu, j \nu } \geqslant 0$. 

Assume that $ \si$ and $\si'$ vanish at  $y$. Then $a_0 =0$ and $a_1$ is computed as follows. Write $\si = f \tau$ and $\si '= f' \tau'$ with $f$ and $f'$ smooth functions vanishing at $y$. Then
\begin{gather} \label{eq:subpairing}
 a_1 = i\frac{d_x f ( X) \overline{ d_x f'(Y)}}{\om ( X, Y)}  \langle \tau (y) , \tau' (y) \rangle_{T_y \Ga, T_y \Ga'} .   
\end{gather}
for any nonvanishing vector $X \in T_x \Ga$ and $Y \in T_x \Ga'$. 
\subsection{Alternative description of Lagrangian states} 

Choose a basis $(\mu, \la)$ of $R$ such that $\om ( \mu, \la ) = 4 \pi$. Denote by $p$ and $q$ the linear coordinates of $E$ dual to this basis. Let $s$ be the section of $L_E$ given by $s = e^{-2i \pi pq}$. We easily compute that 
\begin{gather} \label{eq:ders} 
\nabla s = \frac{4 \pi}{ i} p dq \otimes s .
\end{gather}
Observe that $s( 0) =1$ and $s$ is flat along the lines $\R \mu$, $x + \R \la$ for all $x \in \R \mu$. These conditions completely determine $s$.

Let $q_0$ and $q_1$ in $\R$ be such that $q_0 < q_1 < q_0 + 1$. Let $U$ be the open set  
$ U = \{ [p \mu + q \la ]/ \; q\in ]q_0 , q_1[, \; p \in \R \} $ of $M$.  Let $\phi$ be a smooth real valued function defined on the interval $]q_0, q_1[$. Introduce the submanifold $\Ga$ of $U$ 
$$ \Ga = \{ [ \phi'(q) \mu + q \la ] ;\quad  q \in ]q_0, q_1[ \} $$ 
and the section $\Theta$ of $L_M \rightarrow \Ga$ such that for any $q \in ]q_0 , q_1[$, 
\begin{gather} \label{eq:theta} 
 \Theta ( [ p \mu + q \la ] ) = e^{4i\pi \phi ( q) } s ( p \mu + q\la)  \quad \text{ with } p = \phi' ( q)
\end{gather}  
It follows from (\ref{eq:ders}) that $\Theta$ is flat. Let $\si$ be a section of $\delta_M \rightarrow \Ga$. 

Let $( \Psi_\ell , \ell \in \Z/2k \Z)$ be the basis of $\mathcal{H}_k$ corresponding to $(\mu, \la)$. Let $\xi$ be an admissible family. Denote by $\xi_k( \ell)$, $\ell \in \Z / 2k \Z$ the coefficients of $\xi_k$ in $( \Psi_{-\ell})$.  

\begin{theo} \label{theo:laginbasis}
 The restriction of $\xi$ to $U$ is a Lagrangian state with associated data $(\Ga, \Theta, \si , N)$ if and only if for any $q \in ]q_0 , q_1 [ \cap \frac{1}{2k} \Z$,
\begin{gather} \label{eq:coeff_laginbasis}
 \xi_ k ( 2 k q ) = \Bigl( \frac{k}{2\pi} \Bigr)^{-1/2 +N } e^{ 4 i \pi k \phi (q) }\sum_{\ell =0 }^{\infty} k^{-\ell}   f_{\ell} (q) + \bigo ( k^{-\infty})
\end{gather}
where the $\bigo$ is uniform on any compact set of $]q_0, q_1[$, the $f_{\ell}$'s are smooth functions on $]q_0, q_1[$ and the square of $f_0$ satisfies for any $q$,  
\begin{gather} \label{eq:symb_laginbasis}
 \si^2 ( \phi ' (q) \mu + q \la ) = f_0 ^2 ( q) \frac{ \om ( \cdot, \mu)}{i} .
\end{gather}
\end{theo}

\begin{proof} 
By Proposition 3.2 \cite{LJ1}, the family $(\Psi_0 \in \mathcal{H}_k, k \in \N^*)$ is a Lagrangian state supported by the circle $C= \{ [p \mu] ; \;  p \in \R \} \subset M$ and 
\begin{gather} \label{eq:psi0}
 \Psi_0 ([ p \mu] ) = \Bigl( \frac{k}{2\pi} \Bigr)^{1/4} s^k ( p \mu) \otimes \Om_\mu+ \bigo( k^{-\infty}).
\end{gather}
Assume that $\xi$ is a Lagrangian state. Then we can estimate the scalar product
$$ \xi_k ( 2kq ) = \langle \xi_k , \Psi_{ -2k q } \rangle = \langle \xi_k , T^*_{-q\la } \Psi_0 \rangle$$
with formula (\ref{eq:scalprodlag}). Actually we need a version with parameter of formula (\ref{eq:scalprodlag}) to get a uniform control with respect to $q$. Such a version holds and its proof is not more difficult. Let us explain how we obtain the factor $\exp( 4 i \pi k \phi (q))$ in (\ref{eq:coeff_laginbasis}) and Formula (\ref{eq:symb_laginbasis}) for the leading coefficient. First observe that $\Ga$ intersects $q\la + C$ transversally at the point  $y_q = [\phi ' (q) \mu + q \la ]$.  Translating (\ref{eq:psi0}) by $q \la$, we obtain
$$ \Psi_{ -kq} ( [q \lambda + p \mu ]) = \Bigl( \frac{k}{2\pi} \Bigr)^{1/4} s^k ( q \la + p \mu) \otimes \Om_\mu+ \bigo( k^{-\infty}). $$
By the definition of $\Theta$, cf. Equation (\ref{eq:theta}), we have 
$$ ( \Theta( y_q)^k , s^k  ( y_q) )_{L^k_{M,q}} = e^{4i\pi k \phi ( q)}$$
Equation (\ref{eq:symb_laginbasis}) follows from equation (\ref{eq:pairing}) and the fact that $\Om_{\mu}^2 ( \mu ) =1$. 

Conversely, assume that the asymptotic expansion (\ref{eq:coeff_laginbasis}) holds. By the first part of the proof, there exists a Lagrangian state $\xi '$ such that the coefficients $\langle \xi ' _k , \Psi _{ - 2kq} \rangle$ satisfy the same asymptotic expansion. The coefficients of the sequence $f( \cdot, k)$ in (\ref{eq:lag_state}) have to be defined by successive approximations so that we recover the same coefficients in (\ref{eq:coeff_laginbasis}). Then we have 
$$ \langle  \xi_k - \xi ' _k , \Psi _{ - 2kq} \rangle = \bigo( k^{-\infty})$$ 
uniformly on any compact set of $]q_0, q_1 [$. This has the consequence that the microsupport of $\xi- \xi'$ does not meet $U$. For more details on this last step, see Proposition 2.2 in \cite{oim_torus}.
\end{proof}

\subsection{Application to the functions $\Xi_{m,k}$} \label{sec:appl-funct-xi_m}

Choose a positive basis $( \mu, \la)$ of $R$ and denote by $( \Psi_{\ell}$, $\ell \in \Z / 2k \Z)$ the corresponding basis of $\mathcal{H}_k$. Recall the function $\Xi_{m,k}$ of Section \ref{sec:discr-four-transf}. Define
$$\xi_{m,k} = \sum_{\ell \in \Z / 2k \Z} \Xi _{m,k} \Bigl( \frac{\ell}{k} \Bigr) \Psi_{\ell} . $$
Introduced the 
subsets of $M$ 
\begin{gather} \label{eq:defA12}
 A_1 := \{ [ p \mu ]; \; p \in \R \} , \quad A_2 := \{ [q \la ] , [ \mu /2 + q \la ] ; \; q \in \R \}.
\end{gather}
Introduce the neighborhoods of $A_1 \setminus A_2$ and $A_2 \setminus A_1$ respectively given by $U_1 := A_1 \setminus A_2$ and  $U_2 := A_2 \setminus A_1$.
Let $A := A_1 \cup A_2$ and $ \Theta_A$ be the section of $L_M \rightarrow A$, which is flat and equal to $1$ at the origin. 

\begin{theo} \label{theo:asympt-behav-xi}
The restriction of $( \xi_{m,k}, \; k \in \N^*) $ to $U_1$ (resp. $U_2$) is a Lagrangian state supported by $A_1 \setminus A_2$  (resp. $A_2 \setminus A_1$) with order $-1$ (resp. $1/2+m$) and corresponding section $\Theta_A$. 
\end{theo} 
The symbol can also be computed in terms of the polynomials $P_m $ of Theorem \ref{theo:dev_part}.
\begin{proof} It is a consequence of Theorem \ref{theo:laginbasis} and Theorem \ref{theo:dev_part}.
Indeed, denoting by $\xi_{m,k} (\ell)$ the coefficient of $\Psi_{-\ell}$ in $\xi_{m,k}$, we have for  any $q \in ]0,1[ \cap \frac{1}{2k} \Z$, 
\begin{xalignat}{2}  \notag 
 \xi_{m,k}  ( 2k q ) = &  \Xi _{m} (-2 q)  =  (-1)^m \Xi _{m} (2 q)  \\
\label{eq:singla1}  = & \bigl( (-1)^m + e^{2i k \pi q} \bigr) P_{m} ( k , 2q) 
\end{xalignat}
by Theorem \ref{theo:dev_part}. Recall that $P_{m} (k, q)$ depends smoothly on $q$ (even polynomially) and is polynomial in $k$ with degree $m$. So by Theorem \ref{theo:laginbasis}, the restriction of $(\xi_{m,k})$ to $U_2$ is a Lagrangian state supported by $A_2 \setminus A_1 $. 

To prove the result on $U_1$,  introduce the basis of $\mathcal{H}_k$
$$ \Phi_{\ell} = \frac{e^{i\pi/4}}{\sqrt{2k}} \sum_{n \in \Z/ 2k \Z} e^{i \pi \ell n/k} \Psi_n , \qquad \ell \in \Z/2k \Z$$ 
We check without difficulty that 
$$T^*_{\mu/2k} \Phi_{\ell} = \Phi_{\ell+1}, \qquad T^*_{\la /2k} \Phi_{\ell} = e^{-i \pi \ell/k} \Phi_{\ell+1}.$$ 
Furthermore, the normalization with the factor $e^{i \pi/4}$ has been chosen so that  $\Phi_0 ( 0 ) = \Om_{-\la}$, where $\Om_{-\la} \in \delta $ is such that $\Om _{-\la}^2 ( -\la) = 1$,  cf. Theorem 2.3 of \cite{LJ1} for a proof of this formula. So $( \Phi_{\ell})$ is the basis associated to $( -\la , \mu)$. Furthermore, it follows from the definition of $\Xi_{m,k} $ that 
\begin{xalignat}{2} \notag
\xi_{m,k} = & ( i)^{-m}  \sum_{ n , \ell \in \Z / 2k \Z }  \bigl[ \sin ( \pi \tfrac{\ell}{k} ) \bigr]^{-m} e^{i \pi \ell n /k  } \Psi_{n}\\ \label{eq:xi_phi}
= &  e^{-i \pi/4}  ( i )^{-m} \sqrt{2k}  \sum_{ \ell  \in \Z / 2k \Z }  \bigl[ \sin ( \pi \tfrac{\ell}{k} ) \bigr]^{-m} \Phi_{\ell} 
\end{xalignat} 
where by convention $(0)^{-m} = 0$. So by Theorem \ref{theo:laginbasis}, the restriction of $(\xi_{m,k})$ to $U_1$ is a Lagrangian state supported by $A_1 \setminus A_2 $. 
\end{proof}

\section{Asymptotic behavior of $Z_k ( \Si \times S^1)$ and $Z_k (S)$}
\label{sec:asympt-behav-z_k}

Recall that we introduced in Section \ref{sec:witt-resh-tura} a compact oriented surface $\Si$ with a connected boundary $C$. Let 
\begin{gather} \label{eq:def_E}
E = H_1 ( C \times S^1 , \R), \quad R = H_1 ( C  \times S^1) , \qquad M = E/R.
\end{gather}
Let $\om $ be the symplectic form of $E$ defined as $4\pi$ times the intersection product. Consider the quantum space $\mathcal{H}_k  = H^0 (M, L_M^k ) \otimes \delta$ defined in Section \ref{sec:quantum-spaces}. We denote by $\mathcal{H}_k^{\op{alt}}$ the subspace of alternating sections.

Let $(\mu, \la)$ be the basis of $R$ given by $\mu =[C]$, $\la = [S^1]$.  Denote by $(e_\ell)$ and $(\Psi_\ell)$ the corresponding basis of $V_k ( C \times S^1)$ and $\mathcal{H}_k$ respectively introduced in Section \ref{sec:witt-resh-tura} and Section \ref{sec:quantum-spaces}. We identify $V_k ( \Si \times S^1)$ with $\mathcal{H}_k^{\op{alt}}$ by sending $e_{\ell}$ into $2^{-1/2} ( \Psi_{\ell} - \Psi_{-\ell})$. 
As it was proved in Theorem 2.4 of \cite{LJ1}, this identification depends on the choice of the basis $( \mu, \la)$ only up to a multiplicative factor $\exp (i \pi ( \frac{n}{4} + \frac{n'}{2k}))$, $n $ and $n'$ being two integers independent of $k$. 

\subsection{The state $Z_k ( \Si \times S^1)$}

The vector $Z_k ( \Si \times S^1)$ of $V_k ( C \times S^1)$ is given in the basis $( \Psi_\ell)$ by 
\begin{gather} \label{eq:Z_k1}
 Z_k ( \Si \times S^1) =  \frac{1}{\sqrt 2} \sum_{\ell =1 }^{ k-1} N^{g,k }_{\ell} \bigl( \Psi_{\ell } - \Psi_{-\ell} \bigr)  
\end{gather}
Using Lemma \ref{lem:relat-with-count} and the fact that $\Xi_{2g -1}$ is odd, we get
\begin{gather}  \label{eq:Z_k2}
 Z_k ( \Si \times S^1) =  \frac{C_g }{ \sqrt 2 } k^{g-1} \sum _{ \ell \in \Z / 2k \Z} \Xi _{ 2g -1} \Bigl( \frac{\ell}{k} \Bigr) \Psi_{\ell}  = \frac{C_g }{ \sqrt 2 } k^{g-1} \xi_{2g-1,k}
\end{gather}
where $\xi_{2g-1,k}$ is the vector introduced in Section \ref{sec:appl-funct-xi_m}. By Theorem \ref{theo:asympt-behav-xi},  $(\xi_{2g-1,k})$ is a Lagrangian state, so the same holds for $\bigl( Z_k ( \Si \times S^1) \bigr)$. Let us complete this result by computing the symbol. In the following statement, we use the sets $A_1$, $A_2$ introduced in~(\ref{eq:defA12}), their neighborhoods $U_1$, $U_2$ and the corresponding section $\Theta_A$. 

\begin{theo} \label{theo:asympt-behav-z_k}
The restriction of $( Z_k ( \Si \times S^1) , \; k \in \N^*) $ to $U_1$ is a Lagrangian state with associated data $(A_1 \setminus A_2, \Theta_A, \si _1,  g ) $ where  
$$ \si_1 (p \mu ) \equiv i \sqrt{2} \pi ^g  \bigl[ \sin ( 2 \pi p ) \bigr]^{-2g+1} (dp)^{1/2}, \qquad \forall p \in \R. $$ 
The restriction of $( Z_k ( \Si \times S^1) , \; k \in \N^*) $ to $U_2$ is a Lagrangian state with associated data $(A_2 \setminus A_1, \Theta_A, \si _2,  3g - 3/2) $ where $\si_2$ is given in terms of the function $P_{g,0}$ introduced in Theorem \ref{theo:counting_smoot} by 
$$ \si_2 (q \la )  \equiv  e^{i \pi/4} (\tfrac{\pi}{2})^{1/2} \; P_{g,0} ( 2q) \; ( dq )^{1/2} \  , \qquad \forall q \in (0,1)$$ 
and $\si_2( q \la + \tfrac{1}{2} \mu ) = - \si_2 ( q \la)$. 
\end{theo} 

\begin{proof} 
Introduce the same basis $( \Phi_\ell)$ as in the proof of Theorem \ref{theo:asympt-behav-xi}. 
Denoting by $\eta_k ( \ell)$ the coefficient of $\Phi_{-\ell}$ in $Z_k ( \Si \times S^1)$, we have by Equations (\ref{eq:Z_k2}) and (\ref{eq:xi_phi}) that
\begin{xalignat*}{2}
 \eta_k (2 k p ) =  & \frac{C_g}{\sqrt 2} k^{g-1} e^{i \pi /4} (-1)^{g} \sqrt { 2k}  \bigl[ \sin (- 2 \pi p ) \bigr]^{-2g+1}\\ 
 =  & \frac{e^{i \pi /4}}{2^g} k^{g-1/2} \bigl[ \sin ( 2 \pi p ) \bigr]^{-2g+1} 
\end{xalignat*}
because $C_g = (-1)^{g+1} 2^{-g}$. 
To conclude the computation of $\si_1$, we use that $\om ( \cdot, - \la) /i = 4 i \pi dp $ and equation (\ref{eq:symb_laginbasis}).

Let us compute the symbol $\si_2$. By Theorem \ref{theo:asympt-behav-xi} and Equation (\ref{eq:xi_phi}), the coefficients $\zeta_{k} ( \ell)$ of  $\Psi_{-\ell}$ in $Z_{k} ( \Si \times S^1)$ satisfy 
$$   \zeta_k  ( 2k q ) =  \Bigl( \frac{k}{2 \pi} \Bigr)^{ 3g -2 } \bigl( 1 - e^{2i k \pi q} \bigr) f (q) + \bigo( k^{3g -3}) $$ 
with $f$ a smooth function on $]0,1[$. By Equation (\ref{eq:symb_laginbasis}), we have for any $q \in (0,1)$, 
$$\si_2( q \la ) = f ( q) \sqrt{ 4 i \pi dq }, \qquad \si_2( q \la  + \tfrac{1}{2} \mu ) = - \si_2( q \la).$$  
On the one hand, by Equation (\ref{eq:Z_k1}), $$\zeta_k ( 2k q) = - 2^{-1/2} N_{2kq}^{g,k}.$$ On the other hand, $1- e^{2i k \pi q} =2$ for odd $2k q$. So we conclude from Theorem \ref{theo:counting_smoot} that $2 f( q) = - \frac{1}{\sqrt 2} P_{g,0} ( 2q) $.
\end{proof} 

\subsection{The state $\rho_k ( \varphi ) (Z_k (D \times S^1)) $}\label{sec:remplissage}

Recall that $\varphi $ is a diffeomorphism from $\partial D \times S^1$ to $C \times S^1$. So the homology class $\nu  $ of $\varphi ( \partial D )$ is a primitive vector of $R$, that is $\nu = a \mu + b \la$ where $a,b$ are coprime integers. There is no restriction to assume that $b$ is non negative. Introduce the subset of $M$
\begin{gather} \label{eq:def_B}
 B:= \{ [r \nu] \in M ; \; r \in \R \}  .
\end{gather}
Observe that $B$ is a circle and there is a unique flat section $\Theta_B$ of $L \rightarrow B$ such that $\Theta_B ( 0) = 1$. 
The following result is Theorem 3.3 of \cite{LJ1}.
\begin{theo} \label{theo:state-tore_solide}
 The family $\bigl( \rho_k ( \varphi ) (Z_k (D \times S^1)) $, $k \in \N^*$) is a Lagrangian state with associated data $(B, \Theta_B, \si_B, 0)$ where 
$$ \si_B ( r \nu ) = \sqrt 2 \sin ( 2 \pi r) \Om_\nu $$
with $\Om_\nu \in \delta$ such that $\Om_\nu^2 ( \nu ) = 1$. 
\end{theo} 

For any non vanishing $c$, denote by $I_c$ the interval  $$I_c = ]-\tfrac{1}{2|c|}, \tfrac{1}{2|c|} [.$$  Assume that $a$ and $b$ do not vanish. Consider the open set 
$U = \bigl\{ [ p \mu + q \la ] ; \; p \in I_b, \; q \in I_a \bigr\}$ of $M$. Observe that 
$$ B \cap U = \bigl\{  \bigl[ \tfrac{a}{b} q \mu + q \la \bigr]; \; q \in I_a\}.$$ 
Introduce a function $f_U \in \Ci (M )$ with support contained in $U$, which is identically equal to $1$ on a neighborhood of the origin and such that $f_U(-x) = f_U(x)$. Let $Z( \ell)$ be the coefficients
\begin{gather} \label{eq:4} 
 Z( \ell ) = \bigl\langle f_U \rho_k ( \varphi ) (Z_k (D \times S^1)), e_{\ell} \bigr\rangle, \qquad \ell =1 , \ldots, k-1 
\end{gather}
We deduce from Theorem \ref{theo:laginbasis} and Theorem \ref{theo:state-tore_solide} the following 

\begin{prop}  \label{lem:state-remplissage-e}
We have for any $q \in (0, \tfrac{1}{2} ) \cap \frac{1}{2k } \Z$, 
$$  Z( 2kq)  = \Bigl( \frac{2 \pi}{k}\Bigr)^{1/2} e^{2i \pi k \frac{a}{b} q^2}  \sum_{\ell = 0 }^{\infty} k^{-\ell } f_\ell (q)+ \bigo ( k^{-\infty})$$
where the $f_\ell$ are smooth odd functions on $\R $ with support contained in $I_a$. 
Furthermore $f_0 ( q) =  e^{-i \pi/4} \sin ( 2 \pi q /b)  / \sqrt{ \pi b}$ on a neighborhood of $0$. 
\end{prop}

\subsection{Asymptotics of $Z_k ( S)$} 
Let us assume that $a \neq 0$ and $b \neq 0$. Under this assumption the intersection of $B$ with $A= A_1 \cup A_2 $ is finite. As we will see, each point of $A \cap B$ contributes in the asymptotic expansion of $Z_k (S)$. Actually, since we work with alternating sections, the relevant set is the quotient $X$ of $A \cap B $ by the involution  $-\op{id}_M$
\begin{gather} \label{eq:def_N}
 -\op{id}_M : M \rightarrow M, \qquad  [p \mu + q \la] \rightarrow [- p \mu - q \la]. 
\end{gather}
Let us denote by $N$ the quotient of $M$ by $-\op{id}_M$ so that $X$ is a subset of $N$. 
Introduce the functions $\al$, $\be$ from $N$ to $[0, \pi]$ satisfying
\begin{gather} \label{eq:def_be}
\begin{split} 
 \al ( [ p \mu + q \la ] ) = \arccos ( \cos ( 2 \pi p ) ), \\ 
\be ( [ p \mu + q \la ] ) = \arccos ( \cos ( 2 \pi q ) ) .
\end{split}
\end{gather}
Here it may be worth to observe that $[0,1/2]$ is a fundamental domain for the action of $\Z \rtimes \Z_2$ on $\R$, where $\Z$ acts by translation and $-1 \in \Z_2$ by $- \op{id}_{\R}$. 
The function $\frac{1}{2\pi} \arccos ( \cos ( 2\pi x))$ induces a section from $\R/ \Z \rtimes \Z_2$ to $[0,1/2]$. So if $x = [p \mu + q \la]$, then $\al(x) /2\pi \equiv p$ and $\be(x) /2 \pi \equiv q$ modulo $\Z \rtimes \Z_2$ 

The quotient $N$ is an orbifold with four singular points 
$$ p_1 = [ 0], \quad p_2 = [\mu /2], \quad p_3 = [ \la/2], \quad p_4 = [ \mu/2 + \la /2 ]. $$    
corresponding to the fixed points of $-\op{\id}_M$. All these points belong to $A_2$ and the first two belong to $A_1$ too, actually $A_1 \cap A_2 = \{ p_1 , p_2 \}$.  Since these points play a particular role in the asymptotic expansion of $Z_k (S)$, we divide $X$ into four sets $X_1 = ( A_1 \setminus \{ p_1, p_2 \} ) \cap B$, $X_2 = (A_2 \setminus \{ p_1, p_2, p_3 , p_4 \} ) \cap B$, $ X_3 = \{ p_1, p_2 \} \cap B$ and $ X_4 = \{ p_3, p_4 \} \cap B$. 

\begin{lem} The sets $X_1$ and $X_2$ consist respectively of $\op{E} \bigl( \frac{b-1}{2} \bigr)$ and $|a|-1$ points. $X_3 = \{ p_1 , p_2 \}$ if $b$ is even and $\{ p_1 \}$ otherwise. $X_4 = \{ p_3 \}$ if $a$ is even, $\{ p_4 \}$ if $a$ and $b$ are odd, empty  if $b$ is even.
\end{lem}

\begin{proof} 
Observe that for any $x \in E$, there exists a unique $r \in [0,1/2]$ such that $x = [r ( a \mu + b \la)]$. Furthermore $x \in \{ p_1, p_2, p_3, p_4 \}$ if and only if $r =0$ or $1/2$, $x \in A_1$ if and only if $ r \in \frac{1}{b} \{ 0, 1, \ldots, \op{E} ( b/2) \}$, $x \in A_2 $ if and only if $ r \in \frac{1}{2|a|} \{ 0,1, \ldots , |a|\}$.  We conclude easily. 
\end{proof} 

For any $x$ in $X$, introduce a function $f_x \in \Ci (N)$ which is identically equal to $1$ on a neighborhood of $x$. Assume furthermore that these functions have disjoint supports.  Consider for any $x \in X$ the quantity
$$I_x (k) =  \bigl\langle f_x  Z_k ( \Si \times S^1) , \rho_k ( \varphi) (Z_k (D \times S^1))  \bigr\rangle $$
where the bracket is the scalar product of $\Ci ( M , L^k \otimes\delta_M )$. 
Introduce two integers $c$ and $d$ such that  $ ac + bd = 1$.
\begin{theo} 
We have for any $x \in X_1 \cup X_2 \cup X_4$ that 
\begin{xalignat*}{2} 
 I_x (k) = &  \Bigl( \frac{ k } { 2\pi} \Bigr)^{g - 1/2}      \langle \Theta_A( x) , \Theta_B (x) \rangle ^k \sum_{\ell =0}^{\infty} k^{-\ell} a_{\ell}(x)  + \bigo( k^{-\infty}) 
\end{xalignat*}
where the $a_\ell (x)$ are complex coefficients. If $x \in X_1$ 
$$ n(x) = g - \tfrac{1}{2}, \qquad a_0 (x) \equiv  \frac{2 \pi^{g - 1/2} }{ \sqrt{b}}  \bigl[ \sin ( \al (x) ) \bigr] ^{-2g + 1} \sin \bigl( c \al (x)  \bigr) $$ 
If $x \in X_2$
$$ n(x) = 3g -2 , \qquad a_0 (x) \equiv \frac{1 }{  \sqrt{a}}  P_{g,0} \Bigr( \frac{\be (x) }{\pi} \Bigl) \sin \bigl( d \be(x)   \bigr)  $$
with $P_{g, 0}$ is the function introduced in \ref{theo:counting_smoot}. If $x \in X_4$, 
$$ n(x) = 3g -3 , \qquad a_0 (x) \equiv \frac{i P_{g,0}'(1)  }{ 4 \pi  a^{3/2}}  .$$
\end{theo}

\begin{proof} 
This follows from Theorem \ref{theo:asympt-behav-z_k}, Theorem \ref{theo:state-tore_solide} and Equation (\ref{eq:scalprodlag}). To compute the leading coefficient with equation (\ref{eq:pairing}), we write 
$$\frac{ dp ( \mu) \overline{ \Om_{\nu} ^2 ( \nu) }} { \om ( \mu, \nu)} = \frac{1}{ 4 \pi b}, \qquad   \frac{ dq ( \lambda ) \overline{ \Om_{\nu} ^2 ( \nu) }} { \om ( \la, \nu)} = \frac{-1}{ 4 \pi a} .$$
Furthermore, to compute $\sin ( 2 \pi r)$, we use that for $x = r \nu$,  
$$  r = rac + rbd  \equiv \frac{\al(x)}{2\pi} c  \pm \frac{\be(x)}{2 \pi} d \mod \Z \rtimes \Z_2 $$ 
If $x$ belongs to $X_1$, then $\be ( x) = 0$ which implies that  $\sin ( 2 \pi r) \equiv \sin ( c \al (x))$ up to sign. If $x $ belongs to $X_2$, then $\al ( x) = 0$ or $\pi$ so that  $\sin ( 2 \pi r ) \equiv  \sin ( \be (x) d)$ up to sign. 

To compute $I_x (k)$ with $x \in X_4$, we use formula (\ref{eq:subpairing}).
\end{proof} 
To estimate $I_x (k)$ with $x \in X_3$,  we need the followings results. 
Let $\al \in \R$ and $f \in \Ci_0 ( \R_{+}, \C)$ with $\R_+ = [0, \infty)$. Introduce the sum
$$ S_k^+(f) = \tfrac{1}{2} f( 0 ) + \sum_{\ell =1 }^{\infty} e^{i \frac{\al}{2} \ell^2/k} f \Bigl( \frac{\ell}{k} \Bigr) .$$
$f$ being with compact support, the sum is finite. 

\begin{theo} \label{theo:pd1}
Let $\al \in \R^*$ and $f \in \Ci_0 (\R_+, \C)$ be such that its support is contained in $[0, \frac{2\pi}{|\al|} )$. If $f$ is even and $f(x) = \la x^{2n} + \bigo( x^{2n+1})$, then 
$$ S_k ^+(f) =    k^{\frac{1}{2} - n } \Bigl( \frac{\pi}{2|\al|}  \Bigr)^{\frac{1}{2}} e^{ i \frac{\pi}{4} \op{sgn} \al}   \sum_{\ell= 0 } ^{\infty}  k^{-\ell} c_\ell  + \bigo ( k^{-\infty}) \quad \text{ with } \quad c_0 =  \Bigl( \frac{i}{2 \al} \Bigr)^n \frac{ ( 2n ) !}{ n!} \la $$
and $c_\ell \in \C$ for any positive integer $\ell$.  
If $f$ is odd and $f(x) = \la x^{2n+1} + \bigo( x^{2n+2})$, then
$$ S_k ^+(f) =  k^{-n} \sum_{\ell= 0 } ^{\infty}  k^{-\ell} c_\ell  + \bigo ( k^{-\infty}) \quad \text{ with } \quad c_0 =  \Bigl( \frac{2i} {\al} \Bigr)^{n+1} \frac{ n !}{ 2} \la $$
and $c_\ell \in \C$ for any positive integer $\ell$.
\end{theo}

Here we say that a function of $\Ci ( \R_+)$ is even (resp. odd) if its Taylor expansion at the origin contains only even monomials (resp. odd monomials). 
Similarly, the sum
$$  S_k^{-}(f) = \tfrac{1}{2} f( 0 ) + \sum_{\ell =1 }^{\infty} (-1)^\ell e^{i \frac{\al}{2} \ell^2/k} f \Bigl( \frac{\ell}{k} \Bigr) $$
has the following asymptotic behavior. 
\begin{theo} \label{theo:pd2}
Let $\al \in \R^*$ and $f \in \Ci_0 (\R_+, \C)$ be such that its support is contained in $[0, \frac{\pi}{|\al|} )$. If $f$ is even, $S^{-} _k (f) = \bigo( k^{-\infty})$. If $f$ is odd and $f( x) =  \bigo( x^{n})$, then
$$  S_k^{-}(f) =  k^{-n} \sum_{\ell= 0 } ^{\infty}  k^{-\ell} c_\ell  + \bigo ( k^{-\infty}) $$
for some complex coefficients $c_\ell$. 
\end{theo} 
These two theorems are proved in Section \ref{sec:sing-discr-stat}.
We deduce the following 

\begin{theo} 
For any $x \in X_3$, we have the asymptotic expansion 
\begin{xalignat*}{2} 
  I_x (k)   = &  \Bigl( \frac{ k } { 2 \pi} \Bigr)^{3g - 3}     \sum_{\ell =0 } ^{\infty} a_\ell(x)  k^{-\ell}  + k^{ 2g -3/2}     \sum_{\ell = 0 }^{\infty}  b_\ell(x) k^{-\ell}  + \bigo ( k^{-\infty})   
\end{xalignat*}
with $a_\ell(x)$ and $b_\ell(x)$ complex coefficients, the leading ones being given by 
$$ a_0 (x) =  \frac{  P_{g,0}'(0) }{  4 \pi a^{3/2}}, \qquad b_0 (x)  = e^{i \pi/4} i^g b^{g - 3/2} a ^{-g}  \pi^{-g +1} \frac{\sqrt 2 ( g-1)!}{ ( 2 (  g-1))!}$$
\end{theo}

\begin{proof}
 $I_{p_1} (k)$ is equal to the scalar product of $Z_k  ( \Si \times S^1)$ with the vector $Z_k$ introduced in (\ref{eq:4}) 
$$ Z_k = \sum_{\ell =1 }^{k-1} \bigl\langle f_U \rho_k ( \varphi ) (Z_k (D \times S^1)), e_{\ell} \bigr\rangle e_\ell . $$
 The asymptotic behavior of the coefficients of $Z_k$ is given in Proposition \ref{lem:state-remplissage-e}. 
By Lemma \ref{lem:SiS} and Theorem \ref{theo:counting_smoot}, $Z_k ( \Si \times S^1)$ is the sum of four terms $Z_k^{+,+}$, $Z_k^{+,-}$, $Z_k^{-,+}$, $Z_k^{-,-}$ whose coefficient in the basis $( e_{\ell})$ are 
$$Z_k^{+, \pm} ( \ell ) = \frac{1}{2} \Bigl( \frac{k}{2\pi} \Bigr)^{ 3 g -2 }  \sum _{m = 0 }^{g-1 } k^{-2m} P^{\pm}_{g, m} \Bigl( \frac{\ell}{k} \Bigr) , \quad  Z_k^{-, \pm} ( \ell ) = (-1)^{\ell+1} Z_k^{+, \pm} ( \ell ) .$$
As a consequence of Theorem \ref{theo:pd2}, we have
\begin{gather} \label{eq:5}
 \bigl\langle Z_k , Z_k ^{-, -} \bigr\rangle = \bigo( k^{- \infty}), \qquad \bigl\langle Z_k , Z_k ^{-, +} \bigr\rangle = k^{g - 3/2} \sum_{\ell = 0 }^{\infty} k^{-\ell} c_\ell 
\end{gather}
for some coefficient $c_{\ell}$.  To prove the second formula of Equation (\ref{eq:5}), we have to take into account that 
\begin{gather} \label{eq:2}
P^{+}_{ g,m} (x) = \bigo( x^{ 2 ( g-m  -1)} ).
\end{gather}
 By Theorem \ref{theo:pd1}, we have
 $$ \bigl\langle Z_k , Z_k ^{+, -} \bigr\rangle = \Bigl( \frac{ k } { 2 \pi} \Bigr)^{3g - 3}     \sum_{\ell =0 } ^{\infty} a_\ell(x)  k^{-\ell}  , \qquad \bigl\langle Z_k , Z_k ^{+, +} \bigr\rangle = k^{ 2g -3/2}     \sum_{\ell = 0 }^{\infty}  b_\ell(x) k^{-\ell} 
$$
where $a_0$ and $b_0$ are given by the formula in the statement. To compute $b_0$, we use the expression for $P_{g,0}^{+}$ given in Theorem \ref{theo:counting_smoot}. Furthermore, Equation (\ref{eq:2}) implies that the polynomials $P_{g,m}$ with $m \geqslant 1$ do not enter in the computation.  Since $ 2g - 3/2 > g - 3/2$, $\bigl\langle Z_k , Z_k ^{-, +} \bigr\rangle$ does not contribute to the leading order terms. This concludes the proof for $x= p_1$. 

Assume that $b$ is even. Then $X_3$ consists of $p_1$ and $p_2$ and by a symmetry argument, we see that the computation of $I_{p_3} (k)$ is the same as the one of $I_{p_1} (k)$. Indeed, we have that $(T_{\mu/2}^{*}  + \op{id} ) Z_k (\Si \times S^1)=0$. Furthermore, by Theorem \ref{theo:state-tore_solide},  $T^*_{\nu/2}  \rho_k( \varphi )Z_k ( D \times S^1) =0 $ is a Lagrangian state with associated data $(B, \Theta_B, -\si_B, 0)$. Since $b$ is even, $\mu/2 = \nu/2$. Clearly, $ p_1 + \mu/2 = p_2$, which concludes the proof.   
\end{proof}

\section{Singular discrete stationary phase} \label{sec:sing-discr-stat}

In this section, we prove Theorem \ref{theo:pd1} and Theorem \ref{theo:pd2}.
Let $\al, \be$ be two real numbers. Assume that $\al \neq 0$. Denote by $\R_{+}$ the set of non negative real numbers. For any function $\si \in \Ci_0 ( \R ^{+})$ and positive $\tau$, introduce the sum
$$ S _{\tau}( \si)  = \frac{\si (0)}{2} + \sum_{\ell =1 }^{\infty} e^{i ( \frac{\al}{2} \frac{ \ell^2}{ \tau} - \be \ell)} \si \Bigl( \frac{\ell}{ \tau} \Bigr) $$
In this appendix we study the asymptotics of $S_{\tau} ( \si)$ as $\tau$ tends to infinity. Our treatment is partly inspired by the paper \cite{KeKn}. We will adapt the stationary phase method. The relevant variable is $x = \ell / \tau$. As we will see, the set of stationary points is $\frac{\be}{\al} + 2 \pi \Z$. The origin also contributes non trivially to the asymptotic because the sum starts at $\ell =0$. In Theorem \ref{theo:pd1}, we are in the most delicate situation, because $\be =0$, and the origin is both a stationary point and an endpoint of the summation interval. Let us start with the easiest case where the support of $\si$ does not contain any stationary point. 

\begin{theo}  \label{theo:app1}
For any $\si \in \Ci _0( \R_+)$ such that $\op{Supp} \si \cap \bigl( \frac{\be}{\al} + 2 \pi \Z \bigr) = \emptyset$ and $\si (x) = \bigo( x^n) $ at the origin, we have the following asymptotic expansion
$$ S( \si) = k^{-n} \sum_{\ell =0 }^{ \infty} k^{-\ell} c_{\ell} $$
for some complex numbers $c_{\ell}$. 
\end{theo}

For the proof, we will have to consider more general sums of the form $S_\tau ( \rho ( \cdot, \tau))$ where $\rho ( \cdot, \tau)$, is a family of functions in $\Ci_0 ( \R_{+})$ whose supports are contained in a fixed compact subset of $\R_{+}$ and which admits a complete  asymptotic expansion in inverse power of $\tau$, 
$\rho ( \cdot, \tau) = \rho_0 + \tau ^{-1} \rho_1 + \tau^{-2} \rho_2 + \ldots $, 
for the $\Ci$ topology. We call such a family $(\rho ( \cdot, \tau))$ a {\em symbol}. In particular, for any function $f \in \Ci (\R)$, we will denote by $f (  \frac{1}{\tau} \frac{\partial}{\partial x}) \si$ any symbol with the expansion 
$$ f ( 0 ) \si  + \tau^{-1} f'(0) \si' + \tau^{-2} f^{(2)} (0 ) \si^{(2)} + \ldots $$ 
We will also use the notation $D = \frac{1}{\tau} \frac{\partial}{\partial x}$.

\begin{proof} The sum 
$ \tilde{S}_\tau ( \si) =  \sum_{\ell =0 }^{\infty} e^{i ( \frac{\al}{2} \frac{ \ell^2}{ \tau} - \be \ell)} \si \bigl( \frac{\ell}{ \tau} \bigr) $ satisfies the relation
\begin{gather} \label{eq:1}
 \tilde{S}_{\tau} ( \sigma ( \delta -1 ) ) + \tilde{ S}_{\tau} ( \delta   ( e ^D - 1 ) \si \bigr) + \si ( 0) = 0  
\end{gather}
 where $\delta$ is the symbol $ \delta ( x, \tau) =   e^{ i( \al x - \be ) + i \al /2 \tau }$. To prove this, we 
 apply the summation by part formula
$$ \sum_{\ell = 0 } ^{n} f_{\ell} ( g_{\ell+1} - g_{\ell} ) + \sum_{\ell=0 }^{n} g_{\ell+1} ( f_{\ell + 1} - f_{\ell} ) = f_{n+1} g_{n+1} - f_0 g_0 $$
to the sequences $f_{\ell} = \si \bigl( \frac{\ell}{\tau} \bigr)$ and $g_{\ell} = \exp \bigl( i \frac{\al}{2} \frac{\ell^2}{\tau} - i \be \ell \bigr).$ 
Observe that 
$$ f_{\ell +1} = \si \Bigl( \frac{\ell}{\tau} \Bigr) + \tau^{-1} \si' \Bigl( \frac{\ell}{\tau} \Bigr) + \tfrac{1}{2} \tau^{-2} \si'' \Bigl( \frac{\ell}{\tau} \Bigr) + \ldots = \bigl( e^{D} \si \bigr)  \Bigl( \frac{\ell}{\tau} \Bigr) + \bigo( \tau^{-\infty})$$
so that $f_{\ell+1} - f_{\ell} = \bigl( ( e ^D - 1 ) \si \bigr) ( \ell / \tau)$. Furthermore $g_{\ell +1} = g_{\ell} \delta( \ell/ \tau , \tau) $ and Equation (\ref{eq:1}) follows. 

We have $$\delta (x,\tau) - 1= e^{ \frac{i}{2} ( \al x - \be ) } \sin \bigl( \tfrac{1}{2} ( \al x - \be ) \bigr) + \bigo ( \tau^{-1}).$$ Observe that the zero set of $  \sin \bigl( \frac{1}{2} ( \al x - \be ) \bigr)$ is $ \frac{\be}{\al} + 2 \pi \Z$. So if the support of $\si$ does not intersect this set, we can write $\si = \ga (\delta -1) $ for some symbol $\ga$. Let us apply (\ref{eq:1}) to $\ga$, we obtain 
$$ \tilde{S}_{\tau} ( \si) =  - \si (0) ( \delta (0) -1) + \tau^{-1} \tilde{S}_{\tau} ( \si_1)   $$ 
where $\si_1$ is the symbol $\tau ( \delta ( e^{D} -1) \ga $. Since the support of $\si_1$ is smaller than the support of $\si$, we can do the same computation with $\tilde{S}_{\tau} ( \si_1)$. In this way, we prove that $S_{\tau} ( \si)$ has a complete asymptotic expansion in power of $\tau^{-1}$. With a careful inspection of this computation, we also get that $S_{\tau}  ( \si) = \bigo( k^{-n})$ if $\si$ vanishes to order $k$ at the origin. 
\end{proof}

Choosing $\be = \pi$ in the last result, we obtain Theorem \ref{theo:pd2}. For the proof of Theorem \ref{theo:pd1}, we will use the following relation, which has the advantage to be  more symmetric that Equation (\ref{eq:1}). In the remainder of the appendix, we assume that $\be =0$. 

\begin{lem} \label{lem:sum_part}
For any $\si \in \Ci _0( \R_+)$, we have 
$$   S_{\tau} \bigl( \sin ( \al \cdot ) \si \bigr) = \tfrac{i}{2} \si (0)+  i e^{- i \al / (2 \tau)} \Bigl( S_{\tau}  \bigl( \sinh (D) \si \bigr) + \tfrac{1}{2} ( \cosh (D) \si ) (0)  \Bigr) $$
up to a $\bigo ( \tau^{-\infty})$.  
\end{lem}

\begin{proof} 
We will use the following summation by part formula 
\begin{gather*} 
 \tfrac{1}{2} f_0  \delta_0  (g)    + \sum_{\ell =1 }^{n-1}f_\ell \delta_{\ell} (g) + \tfrac{1}{2} f_n \delta_n (g) +  \tfrac{1}{2} g_0  \delta_0  (f)    + \sum_{\ell =1 }^{n-1}g_\ell \delta_{\ell} (f) + \tfrac{1}{2} g_n \delta_n (f)  + \\
 \tfrac{1}{2} g_0 ( f_1 + f_{-1}) + \tfrac{1}{2} f_0 ( g_1 + g_{-1}) - \tfrac{1}{2} f_{n} ( g_{n-1} + g_{n+1} ) - \tfrac{1}{2} g_n ( f_{n-1} + f_{n+1}) = 0  
\end{gather*}
to the same sequences $f_{\ell}$ and $g_{\ell}$ that we used in the proof of Theorem \ref{theo:app1}.  We have that 
$$ \delta_{\ell} (g) =    2i g_{\ell} \sin \bigl( \al \ell / \tau \bigr) \exp ( i \al / 2 \tau ) , \qquad \tfrac{1}{2} ( g_1 + g_{-1} ) = \exp ( i \al / 2 \tau ). $$
Furthermore
$$ \tfrac{1}{2}  \delta_{\ell} f \equiv  \bigl( \sinh (D) \si \bigr)  \Bigl( \frac{\ell}{\tau} \Bigr), \qquad \tfrac{1}{2} ( f_{1} + f_{-1} ) \equiv \bigl( \cosh (D) \si \bigr) ( 0 ) $$
up to a $\bigo ( \tau^{-\infty})$. 
Applying these expressions in the summation by part formula with $n$ sufficiently large, we get
$$  
2 i e^{i \frac{\al}{2\tau}} S_{\tau} ( \sin ( \al \cdot ) \si ) + 2 S_{\tau} ( \sinh (D) \si ) + \bigl( \cosh (D) \si \bigr) ( 0) + 
\si (0) e^{i \frac{\al}{2 \tau}} \equiv 0 $$
up to a $\bigo( \tau^{-\infty})$,
which was the result to proved. 
\end{proof}

\begin{lem} \label{lem:leading_term}
Let $\rho \in \Ci( \R_+)$ with support contained in $[ 0, \tfrac{2\pi}{ |\al|} )$ and such that $\rho \equiv 1$ on a neighborhood of $0$. Then 
$$ \frac{2}{\tau} S_{\tau} ( \rho ) = \Bigl( \frac{2 \pi } { \tau} \Bigr) ^{ 1/2} \frac{ e^{i \frac{\pi}{4} \op{sgn} \al }}{| \al |^{1/2}}  + \bigo ( \tau ^{-\infty}).$$
\end{lem}
\begin{proof} 
We extend $\rho$ to a smooth even function on $\R$.  Then 
$$2 S_{\tau} ( \rho) = \sum_{\ell= - \infty}^{ + \infty} e^{i \frac{\al}{2} \frac{ \ell^2}{ \tau}} \rho \Bigl( \frac{\ell}{ \tau} \Bigr) $$
By Poisson formula, 
$$ 2 S_{\tau} ( \rho ) = \tau \sum_{\ell = - \infty}^{\infty} I_{\ell}, \qquad \text{ with } \quad I_{\ell} = \int_{\R} e^{ i \tau ( \frac{\al}{2} x^2 - 2 \pi x \ell ) } \rho (x) dx .$$
We can estimate each $I_\ell$ by stationary phase method. For $\ell \neq 0$,  the phase $\frac{\al}{2} x^2 - 2 \pi x \ell$ has a unique critical point $2 \pi \ell / \al$. This point not belonging to the support of $\rho$,   $I_\ell =\bigo ( \tau^{-\infty})$. We can actually prove the stronger result that 
$$ \sum _{\ell \neq 0} I_{\ell}  = \bigo ( \tau^{-\infty} \bigr).$$
Estimating $I_0$ we get the final result. 
\end{proof}

\begin{theo} \label{theo:sing-discr-stat}
Let $\si \in \Ci_0  ( \R_+)$ with support contained in $[0, \tfrac{2 \pi}{ |\al|})$. Then 
$$ S_{\tau} (\si ) = \tau^{1/2} \sum_{\ell = 0 } ^{\infty} a_{\ell} \tau^{-\ell} + \sum_{\ell =0 }^{\infty}  b_{\ell} \tau^{-\ell} $$
where the leading coefficients are
$$ a_0 = \Bigl( \frac{\pi}{2} \Bigr)^{1/2}   \frac{ e^{i \frac{\pi}{4} \op{sgn} \al }}{| \al |^{1/2}} \si(0) , \qquad b_0 = i \frac{\si' (0)}{\al} .$$   
\end{theo}

\begin{proof} By Theorem \ref{theo:app1}, we can assume that the support of $\si$ is contained in $[ 0, \frac{\pi}{|\al|} )$.  Let $\rho$ be a function satisfying the assumption of Lemma \ref{lem:leading_term}. Write  
$$ \si (x) = \si(0)\rho(x)  - i  \sin (\al x) \si_1 ( x ) $$
where  $\si_1 $ is in $\Ci_0 ( \R_+)$ with support in $[0, \tfrac{2\pi}{|\al|})$.
  We have by Lemma \ref{lem:sum_part} 
$$ S_{\tau} ( \si ) =  \si(0) S_{\tau} (  \rho) + \tfrac{1}{2} \si_1 (0) +  e^{- i \frac{\al}{ 2 \tau}} \Bigl( S_{\tau}  \bigl( \sinh (D) \si_1 \bigr) + \tfrac{1}{2} ( \cosh (D) \si_1 ) (0)  \Bigr)$$
By lemma \ref{lem:leading_term},  $\si (0) S_{\tau} ( \rho) = \tau^{1/2} a_0$ where $a_0$ is defined as in the statement. Furthermore $\tfrac{1}{2} \si_1 (0) + \tfrac{1}{2} ( \cosh (D) \si_1 ) (0) = \si_1 (0) + \bigo ( \tau^{-1})$. We also have $\si_1 (0) = b_0$.  So we obtain
$$ S_{\tau} ( \si ) =\tau^{1/2} a_0 + b_0 + \tau^{-1} R_\tau$$
where $R_{\tau}$ is given by 
$$ R_\tau= e^{- i \al / (2 \tau)} \Bigl( S_{\tau}  \bigl(  \tau \sinh (D) \si_1 \bigr) + \tfrac{\tau}{2} \bigl( (\cosh (D) \si_1 ) (0) -  \si_1 (0) \bigr)   \Bigr)
$$
Observe that $\tau \sinh (D) \si_1$ is a symbol, we can apply the same argument to $S_\tau ( \tau \sinh (D) \si_1 )$.
We prove in this way the result  by successive approximations.
\end{proof}

\begin{theo} \label{theo:sing-discr-stat++}
Let $\si \in \Ci_0 (\R_+, \C)$ with support contained in $[0, \frac{2\pi}{|\al|} )$. If $\si$ is even and $\si(x) = \la x^{2n} + \bigo( x^{2n+1})$ at the origin, then 
$$ S_\tau (f) =    \tau^{\frac{1}{2} - n } \Bigl( \frac{i \pi}{2\al} \Bigr)^{\frac{1}{2}}   \sum_{\ell= 0 } ^{\infty}  \tau^{-\ell} c_\ell  + \bigo ( \tau^{-\infty}) \quad \text{ with } \quad c_0 =  \Bigl( \frac{i}{2 \al} \Bigr)^n \frac{ ( 2n ) !}{ n!} \la .$$
If $\si$ is odd and $\si(x) = \la x^{2n+1} + \bigo( x^{2n+2})$, then
$$ S_\tau (f) =  \tau^{-n} \sum_{\ell= 0 } ^{\infty}  \tau^{-\ell} c_\ell  + \bigo ( \tau^{-\infty}) \quad \text{ with } \quad c_0 =  \Bigl( \frac{2i} {\al} \Bigr)^{n+1} \frac{ n !}{ 2} \la. $$
\end{theo}   

\begin{proof} 
First, by adapting the proof of Theorem \ref{theo:sing-discr-stat}, we show that if $\si$ is even, the coefficients $b_\ell$ vanish, whereas if $\si$ is odd, the coefficient $a_\ell$ vanish. For instance, if $\si$ is even, $\si_1$ is odd, so that $\si_1 ( 0) = 0$ and $\sinh (D) \si_1$ is even. We conclude by iterating.

To compute the leading coefficients, we use the filtration $\bigo ( m)$, $m\in \N$ of the space of symbols defined as follows:  
$$f \in \bigo (m) \Leftrightarrow f = \sum_{0 \leqslant \ell \leqslant m/2} \tau^{-\ell} g_{\ell} + \bigo ( \tau^{-m/2}  )$$ where for any $\ell$, the coefficient $g_{\ell} \in \Ci ( \R_+)$ vanishes to order $m -2\ell$ at the origin. Observe that  if $f \in \bigo ( m+1)$ then $f ( 0 ) = \bigo ( \tau ^{- ( m+1) / 2} )$ and $Df \in \bigo ( m+1)$. 

Assume that $\si \in \bigo (m)$ and that we want to compute $S_{\tau} (\si)$ up to a $\bigo(\tau ^{- m/2})$. 
We consider again the proof of Theorem \ref{theo:sing-discr-stat}. Introduce the function
$$ \ga (x) = (\si (x) - \si (0) \rho (x) )/x .$$
Then $ \si_1 = \frac{i}{\al} \ga +  \bigo( m+1)$ and $\sinh (D) \si_1 = \frac{i }{\al} D \ga +  \bigo( m+1)$. From this, we deduce that 
$$ S_{\tau} ( \si ) = S_{\tau} ( \rho ) \si (0) + \tfrac{i}{\al}  \ga (0) + \tfrac{i }{\al}   S_{\tau} \bigl( D \ga + \bigo ( m+1) \bigr)  + \bigo(\tau ^{- m/2})$$
To conclude the proof, we choose $\si = \la x^m \rho$ and apply this formula as many times as necessary. 
\end{proof}

This completes the proof of Theorem \ref{theo:pd1}.

\section{Geometric interpretation of the leading coefficients} 
\label{sec:geom-interpr-lead}

\subsection{Symplectic volumes} \label{sec:symplectic-volumes}

For any $t \in [0,1]$, denote by $\mathcal{M} (\Si, t)$ the moduli space of flat $\op{SU}(2)$-principal bundles whose holonomy $g$ of the boundary $C = \partial \Si$ satisfies $\frac{1}{2}\op{tr} (g)  =   \cos ( 2 \pi t)$. Equivalently, $\mathcal{ M} ( \Si , t)$ is the space of conjugacy classes of group morphisms $\rho$ from $\pi_1 ( \Si)$ to $\op{SU}(2)$ such that for any loop $\ga \in \pi_1(\Si)$ isotopic to $C$, $\frac{1}{2}\op{ tr} ( \rho ( \ga)) =   \cos ( 2 \pi t)$. 

We say that a morphism $\rho$ from $\pi_1 ( \Si)$ to $\op{SU}(2)$ is irreducible if the corresponding representation of $\pi_1 ( \Si)$ in $\C^2$ is irreducible. 
The subset $\mo ^{\op{irr}} ( \Si, t)$ of $\mo ( \Si, t)$ consisting of conjugacy classes of irreducible morphisms is a smooth symplectic manifold.  Using the usual presentation of $\pi_1 ( \Si)$, one easily sees that any morphism $\rho : \pi_1 ( \Si )  \rightarrow \op{SU}(2)$ such that $\rho ( \ga) \neq \op{id} $ for $\ga$ isotopic to $C$, is irreducible. Consequently $\mo ^{\op{irr} } ( \Si , t ) = \mo ( \Si , t)$ for $t \in ( 0,1]$.  The subset of $\mo ( \Si , 0)$ consisting of non irreducible representation is in bijection with $\op{Mor} ( \pi_1 ( \Si ), \R / \Z)$, the set of group morphisms from $\pi_1 ( \Si)$ to $\R/Z$. 

Furthermore, for $t \in (0,1)$,  $\mo ( \Si , t)$ is $2 ( 3g -2 )$-dimensional, whereas $\mo ( \Si , 1) $ and $\mo ^{\op { irr}} ( \Si , 0)$ have dimension $2 ( 3g -3)$.

\begin{theo} \label{theo:RR}
For any $ k, \ell \in \N$ such that $ 0< \ell \leqslant k$ and $\ell$ is even, we have 
$$ N^{g, k  + 2 }_{\ell +1} = \int_{ \mo ( \Si , s)} e^{  \frac{k}{2 \pi} \om_s } \op{Todd}_s $$
where $s = \ell / k$, $\om_s$ is the symplectic form of $\mo ( \Si, s)$ and $\op{Todd}_s$ any representant of its Todd class  
\end{theo}

As a corollary, we can compute the polynomial function $P_{g,0}$ as a symplectic volume. 

\begin{cor} For any $s \in ( 0,1)$, we have 
$$  P_{g,0} ( s) = \int_{\mo ( \Si, s)}  \frac{ \om_s ^{3g-2}}{ ( 3g -2)!} .$$
\end{cor}

We can actually recover partially  Theorem \ref{theo:counting_smoot} in this way. Introduce the space $\mo ( \Si)$ of conjugacy classes of morphisms from $\pi_1( \Si)$ to $\op{SU}(2)$. Let $f : \mo ( \Si) \rightarrow \R$ be the function sending $\rho$ into $\frac{1}{\pi} \arccos ( \op{tr }( \rho (C)))$. Then for each $s \in [0,1]$, $\mo ( \Si , s)$ is the fiber at $s$ of $f$. Furthermore $(0,1)$ is the set of regular values of $f$. So we can identify the $\mo ( \Si , s)$, $ s\in (0,1)$ to a fixed manifold $F$, by a diffeomorphism uniquely defined up to isotopy. In particular, the homology groups of $\mo ( \Si , s)$ are naturally identified with the ones of $F$. 

In \cite{Je2}, Jeffrey introduced an extended moduli space $\mo ^{ \mathfrak{t}} ( \Si)$. This space is a $2(3g-1)$-dimensional $( \R / \Z)$-Hamiltonian space, such that for any $s \in ( 0,1)$, $\mo ( \Si ,s )$ is the symplectic reduction of $\mo ^{\mathfrak{t}} ( \Si)$ at level $s$. We recover that the various $\mo ( \Si, s)$, $s \in (0,1)$ can be naturally identified with a fixed manifold $F$ up to isotopy. 

Furthermore, by Duistermaat-Heckman Theorem \cite{DuHe}, the cohomology class of $\om_s$ is an affine function of $s$ with value in $H^2(F)$, that is $[\om_s] = \Om + s c$, where $\Om$ and $c$ are constant cohomology classes in $H^2(F)$.
 This implies in particular that $P_{g,0}$ is polynomial with degree $(3g -2)$.

We can explain in this way why the shifts of $\ell$ and $k$ we introduced are natural. Indeed it has been proved by Meinrenken-Woodward \cite{MeWo2} that the canonical class $c_1$ of $\mo ( \Si , s)$ is $-4 \Om -2 c$. Using that $\op{Todd} = \hat{A} e^{-\frac{1}{2} c_1} $ where $\hat{A}$ is the $A$-genus, we obtain the that for any $k, \ell \in \N$ such that $ 0< \ell < k$ and $\ell$ is odd, 
$$  N^{g,k}_{\ell} = \int_{\mo ( \Si , s)} e^{  \frac{k}{2 \pi} \om_s } \hat{A}, \qquad \text{ with } s = \ell / k . $$
Since the $\hat{A}$-genus belong to $\bigoplus_{\ell} H^{4\ell}(F) $, it follows that
$$ Q( k ,s) = \int_{\mo ( \Si , s)} e^{  \frac{k}{2 \pi} \om_s } \hat{A} $$ 
is  a linear combination of the monomial $k^{2m} s ^p$ with $0 \leqslant p \leqslant 2m$ and $0 \leqslant m \leqslant g-1$, which was already proved in Theorem \ref{theo:counting_smoot}. 

\subsection{Character varieties} 

For any topological space $V$, introduce the character variety $\mathcal{M} ( V) $ defined as the space of group morphisms from $\pi_1 (V)$ to $\op{SU}(2)$ up to conjugation. 
If $W$ is a subspace of $V$, we have a natural map from $\mathcal{M} (V)$ to $\mathcal{M} (W)$, that we call the restriction map. 

For the circle $C$, $\pi_1 (C)$ being cyclic, $\mathcal{M} (C)$ identifies with the set of conjugacy classes of $\op{SU} (2)$. So $\mathcal{M} (C) \simeq [0,\pi]$ by the map sending the morphism $\rho$ to the number $\arccos ( \frac{1}{2} \op{tr} \rho (C) )$. Similarly, $\mo (S^1) \simeq [0,\pi]$. 

For the two-dimensional torus $C \times S^1$, there is a natural bijection between $\mathcal{M}(C \times S^1)$ and the quotient of $H_1 (C \times S^1, \R)$ by $H_1(C \times S^1) \rtimes \Z_2$ defined as follows. Identify $\pi_1 ( C \times S^1)$ with $H_1 ( C \times S^1)$ and denote by $\cdot$ the intersection product of $H_1( C \times S^1)$. Then to any $x \in H_1 ( C \times S^1, \R)$ we associate the representation $\rho_x$ given by 
\begin{gather} \label{eq:defrhox}
 \rho _x ( \ga) = \exp ( (x \cdot \ga) D), \qquad \forall \ga \in H_1( C \times S^1)\end{gather}  
where  $D \in \op{SU}(2)$ is the diagonal matrix with entries $2i \pi$, $-2 i\pi$. 

Recall that we denote by $M$ the quotient of $H_1 ( C \times S^1 , \R)$ by $H_1 ( C \times S^1)$ and by $N$ the quotient of $M$ by $- \op{id}_M$, cf. (\ref{eq:def_E}) and (\ref{eq:def_N}). So the map sending $x$ to $\rho_x$ induces a bijection between $N$ and $\mathcal{M} ( C \times S^1)$. Furthermore the restriction maps from $\mathcal{M} ( C \times S^1)$ to $\mathcal{M} ( C)$ and $\mathcal{M} ( S^1)$ identify respectively with the maps $\be$ and $\al$ introduced in (\ref{eq:def_be}). 

Recall that we introduced subsets $A_1$, $A_2$, $A= A_1 \cup A_2$ and $B$ of $M$. We denote by $\tilde{A}_1$, $\tilde{A}_2$, $\tilde{A}$ and $\tilde B$ their projections in $N$. So $\tilde{A}_1$ and $\tilde{A}_2$ consists respectively of the classes $[\rho] \in \mo ( C\times S^1)$ such that $\rho (C) = \id$ or $\rho ( S^1) = \pm \id$.  In other words $\tilde{A}_1  = \beta^{-1} (0)$ and $\tilde{A}_2 = \alpha^{-1} ( \{ 0,\pi \} )$.

\begin{lem} \label{lem:mapf}
The image of the restriction map $f$ from $\mo ( \Sigma \times S^1)$ to $\mo ( C \times S^1)$ is $\tilde{A}$. For any $x \in \tilde{A}_1 \setminus \tilde{A}_2$, $f^{-1}(x)$ identifies with $\op{Mor} ( \pi_1 ( \Si), \R /\Z)$. For any $x \in \tilde{A}_2$, $f^{-1} (x)$ identifies with $\mathcal{M} ( \Si, \be (x) /\pi)$.
\end{lem}  

\begin{proof} Let $\mathbb{T}$ be the subgroup of $\op{SU}(2)$ consisting of diagonal matrices. So $\mathbb{T} \simeq \R / \Z$. We will use that for any $g \in \mathbb{T} \setminus \{ \pm \op{id} \}$, the centralizer of $g$ in $\op{SU}(2)$ is $\mathbb{T}$. 

Let $\rho $ be a morphism from $\pi_1 ( \Si \times S^1) =\pi_1 ( \Si ) \times \pi ( S^1)$ to $\op{SU}(2)$. The restriction $\rho'$ of $\rho $ to $\Si$ commutes with $\rho (S^1)$. Conjugating $\rho$ if necessary, $ g = \rho ( S^1)$ belongs to $\mathbb {T}$. Consider the following two cases: 
\begin{itemize} 
\item If $g$ is not central, the image of $\rho'$ is contained in $\mathbb{T}$. This implies that $\rho ( C) = \op{ \id}$ so that $f ( \rho ) \in \tilde{A}_1$. Conversely, any $g \in \mathbb{T} \setminus \{ \pm \op{id} \}$ and $\rho ' \in \op{Mor} ( \pi_1 ( \Si), \mathbb{T} )$ determines a unique $\rho \in \mathcal{M} ( C \times S^1)$. 
\item If $g$ is central, then $f( \rho) \in \tilde{A}_2$. 
Conversely, any $g = \pm \id$ and $\rho' \in \mathcal{M} ( \Si )$ determine a unique $\rho \in \mathcal{M} ( C \times S^1)$. 
\end{itemize} 
To end the proof in the second case,  we view $\mathcal{M} ( \Si )$ as the union of the $\mathcal {M} ( \Si , t) $ where $t$ runs over $[0,1]$. 
\end{proof} 

Recall that $S$ is the Seifert manifold obtained by gluing the solid torus $D \times S^1$ to $\Sigma \times S^1$ along the diffeomorphism $\varphi$ of $C \times S^1$. Furthermore, $X = \tilde{A} \cap \tilde{B}$. 
\begin{theo} 
The components of $\mo (S)$ are in bijection with $X$. For any $x \in X$ the corresponding component is homeomorphic with $f^{-1} (x)$. 
\end{theo} 

\begin{proof}  It follows from Van Kampen Theorem that $\pi_1 (S)$ is the quotient of $\pi_1 ( \Si \times S^1)$ by the subgroup generated by $\varphi (C) $. So the group morphisms from $\pi_1 (S) $  identify with the group morphisms from $\pi ( \Si \times S^1)$ sending $\varphi ( C)$ to the identity. 

On the other hand, for any $x \in H_1 ( C \times S^1, \R)$ the corresponding representation $\rho_x$ defined in (\ref{eq:defrhox}) is trivial on $\varphi(C)$ if and only if $x$ belongs to the line generated by $ \nu = a \mu + b \la$. So $\tilde{B}$ consists of the conjugacy classes of representations which are trivial on $\varphi(C)$.  

This implies that the restriction map from $\mo (S)$ to $\mo ( \Si \times S^1)$ is injective, and its image is $f^{-1} (\tilde{B})$. The conclusion follows from Lemma \ref{lem:mapf}, taking into account that the fibers of $f$ are connected. 
\end{proof} 

\subsection{Chern-Simons invariant} 

For any three-dimensional closed oriented manifold $V$ and $\rho \in \mo (V)$ the Chern-Simons invariant of $\rho$ is defined by
\begin{gather} \label{eq:defCS}
 \op{CS}( \rho) = \int _V \tfrac{2}{3} \al^3 + \al \wedge d\al \in \R / 2 \pi \Z
\end{gather} 
 where $\al \in \Om ^1(V, \mathfrak{su} (2))$ is any connection form whose holonomy representation is $\rho$.

\begin{theo} 
For any $ \rho \in \mo ( S)$, the Chern-Simons invariant of $\rho$ is given by
$$ e^{ i \op{CS} ( \rho ) } = \bigl\langle \Theta_A (x) , \Theta_B (x) \rangle  $$
where $x  \in \mo ( C \times S^1)$ is the restriction of $\rho $ to $C \times S^1$. 
\end{theo} 

The proof is based on the relative Chern-Simons invariants introduced in \cite{RaSiWe}, cf. also \cite{Fr}.  
\begin{proof}[proof (sketch)] We can define a relative Chern-Simons invariant for compact oriented 3-manifold $V$ with boundary. To do this we define first a complex line bundle $L \rightarrow \mathcal{M} ( \partial V)$, called the Chern-Simons bundle. Then for any $\rho \in \mo ( V)$, $e^{i \op{CS} (\rho)}$ is by definition a vector in $L_{r(\rho)}$ where $r$ is the restriction map from $\mo ( V)$ to $\mo ( \partial V)$. This invariant has the three following properties:
\begin{itemize} 
\item The fiber of $L$ at the trivial representation has a natural trivialization. If $\rho \in \mo (V)$ is the trivial representation, then  $e^{ i \op{CS} ( \rho)}  =1$ in this trivialization.
\item $L$ has a natural connection, and the section of $r^* L$ sending $\rho$ into $e^{i \op{CS} ( \rho)}$ is flat. 
\item If $V$ is closed and obtained by gluing two manifolds $V_1$ and $V_2$ along the common boundary, then for any $\rho \in \mo (V)$, 
$$e^{i \op{CS} ( \rho) } = \bigl\langle e^{i\op{CS} ( \rho_1)}, e^{i \op{CS} ( \rho_2)} \bigr\rangle$$ where $\rho_1$ and $\rho_2$ are the restrictions of $\rho$ to $V_1$ and $V_2$ respectively. 
\end{itemize} 
In our case, the pull-back of the Chern-Simons bundle of $\mo ( C \times S^1)$ by the projection $ M \rightarrow \mo ( C \times S^1)$ is the prequantum bundle $L_M$, cf. \cite{LJ2}. Furthermore the image of the restriction maps from $\mo ( D \times S^1)$ and $\mo (\Si \times S^1)$ to $\mo ( C \times S^1)$ are respectively $\tilde{A}$ and $\tilde{B}$. We conclude by lifting everything to $M$ and by using that $\Theta_A$ and $\Theta_B$ are flat and satisfy $\Theta_A ( 0) = \Theta_B (0)$.  
\end{proof}

\bibliographystyle{alpha}
\bibliography{biblio}

\end{document}